\documentclass[a4paper,12pt]{article}
\usepackage{latexsym,amsmath,amsthm,amssymb}
\usepackage{a4wide}
\usepackage{hyperref}
\usepackage{marginnote}
\usepackage{color}
\usepackage{cite}
\hypersetup{
pdftitle={Adams inequality}   
pdfauthor={Van Hoang Nguyen},
colorlinks = true,
linkcolor = magenta,
citecolor = blue,
}

\theoremstyle{plain}
\newtheorem{theorem}{Theorem}[section]

\newtheorem{proposition}[theorem]{Proposition}
\newtheorem{lemma}[theorem]{Lemma}

\theoremstyle{definition}

\theoremstyle{remark}

\renewcommand{\thefootnote}{\arabic{footnote}}
\newcommand{\C}[1]{\ensuremath{{\mathcal C}^{#1}}} 

\def\R{\mathbb R}
\def\C{\mathbb C}

\def\al{\alpha}
\def\om{\omega}
\def\Om{\Omega}
\def\de{\delta}
\def\De{\Delta} 

\def\si{\sigma}
\def\lam{\lambda}
\def\vphi{\varphi}
\def\ep{\epsilon}
\def\na{\nabla}
\def\pa{\partial}
\def\la{\langle} 
\def\ra{\rangle} 
\def\lt{\left}
\def\rt{\right}
\def\o{\overline}

\def\i0i{\int_0^\infty}

\numberwithin{equation}{section}

\title{A sharp Adams inequality in dimension four and its extremal functions}
\author{Van Hoang Nguyen\footnote{
Institut de Math\'ematiques de Toulouse, Universit\'e Paul Sabatier, 118 Route de Narbonne, 31062 Toulouse c\'edex 09, France.}
}

\begin{document}
\maketitle


\renewcommand{\thefootnote}{}

\footnote{Email: \href{mailto: Van Hoang Nguyen <van-hoang.nguyen@math.univ-toulouse.fr>}{van-hoang.nguyen@math.univ-toulouse.fr}}

\footnote{2010 \emph{Mathematics Subject Classification\text}: 46E35.}

\footnote{\emph{Key words and phrases\text}: Adams inequality, blow-up analysis, sharp constant, extremal functions, regularity theory.}

\renewcommand{\thefootnote}{\arabic{footnote}}
\setcounter{footnote}{0}

\begin{abstract}
Let $\Om$ be a smooth oriented bounded domain in $\R^4$, $H_0^2(\Om)$ be the Sobolev space, and $\lambda_1(\Om)= \inf \{\|\Delta u\|_2^2 : u\in H_0^2(\Om), \|u\|_2 =1\}$ be the first eigenvalue of the bi-Laplacian operator $\Delta^2$ on $\Om$. For $\alpha \in [0,\lambda_1(\Om))$, we define $\|u\|_{2,\alpha}^2 = \|\Delta u\|_2^2 - \alpha \|u\|_2^2$, for $u \in H_0^2(\Om)$. In this paper, we will prove the following inequality
\[
\sup_{u\in H_0^2(\Om),\, \|u\|_{2,\alpha} \leq 1} \int_{\Om} e^{32 \pi^2 u(x)^2} dx < \infty.
\]
This strengthens a recent result of Lu and Yang \cite{LuYang}. We also show that there exists a function $u^*\in H_0^2(\Om)\cap C^4(\o{\Om})$ such that $\|u^*\|_{2,\alpha} =1$ and the supremum above is attained by $u^*$. Our proofs are based on the blow-up analysis method.

\end{abstract}

\section{Introduction}
Let $\Om$ be a smooth bounded domain in $\R^n$. The Sobolev inequality says that the embedding $W_0^{k,p}(\Om) \hookrightarrow L^{\frac{np}{n-kp}}(\Om)$ holds if $p < n/k$, where $W_0^{k,p}(\Om)$ denotes the Sobolev space of functions vanishing on boundary $\pa \Om$ together their derivatives of order less than $k-1$. Such inequality plays an important role in many branch of mathematics such as analysis, geometric, partial differential equations, calculus of variations, etc.  However, when $p =n/k$ the embedding  $W_0^{k,n/k}(\Om) \hookrightarrow L^{\infty}(\Om)$ does not holds. In this case, the Moser--Trudinger and Adams inequalities are perfect replacement. The Moser--Trudinger inequality was established independently by ${\rm Yudovi\check{c}}$ \cite{Y1961}, ${\rm Poho\check{z}aev}$ \cite{P1965} and Trudinger \cite{T1967}. This inequality was sharpened by Moser \cite{M1970} by finding its sharp constant. This sharp form asserts that the existence of a constant $C_0 >0$ such that
\begin{equation}\label{eq:MT}
\frac1{|\Om|} \int_\Om \exp(\beta |f(x)|^{\frac n{n-1}}) dx \leq C_0,
\end{equation}
for any $\beta \leq \beta_0 = n \om_{n-1}^{1/(n-1)}$ where $\om_{n-1}$ denotes the surface area of the unit sphere of $\R^n$, for any bounded domain $\Om$ and for any function $f\in W_0^{1,n}(\Om)$ with $\|\na f\|_n \leq 1$. If $\beta >\beta_0$ then the above inequality does not hold with uniform $C_0$ independent of $f$. Moser--Trudinger is a crucial tool in studying the partial differential equation inequality with exponential nonlinearity. Because of its importance, there are many generalization of Moser--Trudinger inequality, such as Moser--Trudinger inequality on Heisenberg group, on complex sphere or on compact Riemannian manifold \cite{CohnLu2001,CohnLu2004,Li2001}. It was also extended to entire Euclidean space by Ruf \cite{Ruf2005} for dimension two and by Li and Ruf \cite{LiRuf2008} for any dimension or entire Heisenberg group by Lam and Lu \cite{LamLu2012*}, or on hyperbolic space by Wang and Ye \cite{WY2012}. In \cite{TZ2000}, Tian and Zhu proved a Moser--Trudinger type inequality for alomost plurisubharmonic functions on any K\"ahler-Einstein manifolds with positive curvature.

The existence of the extremal function for Moser--Trudinger inequality was first proved by Carleson and Chang \cite{CC1986} for the unit ball in $\R^n$. In \cite{Flucher1992}, Flucher proved the existence of extremal function for Moser--Trudinger inequality for any smooth domain in $\R^2$. This result was then extended to any dimension by Lin \cite{Lin1996}. The existence of extremal function for Moser--Trudinger inequality on compact Riemannian manifold was studied by Li \cite{Li2005}. We refer the reader to \cite{Csato2015,Csato2016,doO2014,doOSouza,LiN2007,LiRuf2008,Ruf2005,WY2012,Yang2006,Yang2006a,Yang2015,Yang2017,Yang2017a} for more existence results of extremal functions for Moser--Trudinger type inequalities.

Suggesting by the concentration--compactness principle due to Lions \cite{Lions1985}, Adimurthi and Druet established in \cite{AD2004} the following generalization of Moser--Trudinger inequality on any bounded domain $\Om\subset \R^2$
\begin{equation}\label{eq:AdiDruet}
\sup_{u\in H_0^1(\Om),\, \|\na u\|_2 \leq 1} \int_\Om e^{4\pi(1 + \alpha \|u\|_2^2) u^2} dx < \infty,
\end{equation}
for any $0 \leq \alpha < \lam(\Om)$, where $\lam(\Om) = \inf_{u\in H_0^1(\Om),\, \|u\|_2 \leq 1} \|\na u\|_2^2$ is the first eigenvalue of Laplace operator $-\Delta$. The existence of extremal function for \eqref{eq:AdiDruet} was proved by Yang in \cite{Yang2006a}. This result was extended by Yang \cite{Yang2006,Yang2007} to the cases of high dimension and compact Riemannian surfaces, by Lu and Yang \cite{LY2009} and Zhu \cite{Zhu2014} to the version of $L^-$norm, by Souza and do \'O \cite{doO2014,doOSouza} to the whole Euclidean space, and by Tintarev \cite{Tintarev} to the following form
\begin{equation}\label{eq:Tintarev}
\sup_{u\in H_0^1(\Om), \, \|\na u\|_2^2 -\alpha \|u\|_2 \leq 1} \int_\Om e^{4\pi u^2} dx < \infty,
\end{equation}
with $0\leq \alpha < \lambda(\Om)$. Evidently, \eqref{eq:Tintarev} implies \eqref{eq:AdiDruet}. In \cite{Yang2015}, Yang generalized \eqref{eq:Tintarev} to the cases that large eigenvalues are involved, as well as to the manifold case. The existence of extremal functions for \eqref{eq:Tintarev} also obtained in \cite{Yang2015}. In \cite{Yang2017}, Yang and Zhu studied the singular version of \eqref{eq:Tintarev}. They proved the existence of extremal functions for the following singular Moser--Trudinger inequality
\begin{equation}\label{eq:sigularTintarev}
\sup_{u\in H_0^1(\Om), \, \|\na u\|_2^2 -\alpha \|u\|_2 \leq 1} \int_\Om \frac{e^{4\pi(1-\beta) u^2}}{|x|^{2\beta}} dx < \infty, \quad \alpha < \lambda(\Om)
\end{equation}
where $\Om$ is a smooth bounded domain in $\R^2$ containing the origin in its interior and $0 \leq \beta < 1$. The same existence result for the singular Moser--Trudinger inequality on whole Euclidean space was recently proved by Yang and Zhu in \cite{Yang2017a}.

Adams inequality is the version of higher order of derivatives of Moser--Trudinger inequality. The study of this inequality was started by the work of Adams \cite{Adams1988}. To state Adams inequality, we use the symbol $\na^mu$ with $m$ is a positive integer, to denote the $m^{\rm th}$ order gradient for $u\in C^m$, the class of $m^{\rm th}$ order differentiable functions,
\[
\na^m u = 
\begin{cases}
\Delta^{m/2} u&\mbox{if $m$ even,}\\
\na\Delta^{(m-1)/2} u&\mbox{if $m$ odd,}
\end{cases}
\]
where $\na$ and $\De$ denotes the usual gradient operator and usual Laplacian respectively. Adams proved in \cite{Adams1988} that for any positive integer $m$ less than $n$, there exists a constant $C_0(n,m)$ such that for any bounded domain $\Om \subset \R^n$, it holds
\begin{equation}\label{eq:Adams}
\sup_{u\in W_0^{m,\frac nm}(\Om),\, \|\na^m u\|_{\frac nm} \leq 1} \frac{1}{|\Om|} \int_\Om \exp(\beta |u|^{\frac n{n-m}}) dx \leq C_0(n,m),
\end{equation}
for any $\beta \leq \beta(n,m)$ with
\[
\beta(n,m) =
\begin{cases}
\frac n{\om_{n-1}} \lt[\frac{\pi^{n/2} 2^m \Gamma(\frac{m+1}2)}{\Gamma(\frac{n-m+1}2)}\rt]^{\frac n{n-m}} &\mbox{if $m$ odd,}\\
\frac n{\om_{n-1}} \lt[\frac{\pi^{n/2} 2^m \Gamma(\frac{m}2)}{\Gamma(\frac{n-m}2)}\rt]^{\frac n{n-m}} &\mbox{if $m$ even.}
\end{cases}
\]
Furthermore, for $\beta > \beta(n,m)$ the supremum above will be infinite. Notice that when $m=1$, \eqref{eq:Adams} reduces to Moser--Trudinger inequality \eqref{eq:MT}.

Remark that the work of Moser and of Carleson and Chang was based on the rearrangement argument to reduce problem to the one-dimensional problem. However, we can not adapt this symmetrization technique in the case $m\geq 2$ since we do not know whether the $L^{\frac nm}$ norm of the $m^{\rm th}$ gradient of a function decreases under the rearrangement operator. In order to establish \eqref{eq:Adams}, Adams use the representation of $u$ in terms of its gradient function $\na^m u$ using a convolution operator, and then apply O'Neil's idea \cite{O'Neil1963} of rearrangement of convolution of two functions together with the idea which originally goes back to Garcia. Such an argument avoids in dealing with the issue of $L^{\frac nm}$ norm preserving of the gradient of the rearranged functions. This idea has also been developed to derive the sharp Adams inequality on Riemannian manifolds without boundary by Fontana \cite{Fontana1993}, on the measure spaces by Fontana and Morpurgo \cite{FM2011}. The sharp Adams inequality was also generalized to whole Euclidean space in the works of Fontana and Morgurgo \cite{FM2015}, of Lam and Lu \cite{LamLu2012d,LamLu2013} and of Ruf and Sani \cite{RufSani2013}. The sharp Adams inequality was recently established on the hyperbolic spaces by Karmakar and Sandeep \cite{KS2016}. 

It remains an open problem whether Adams inequality has an extremal function. Unlike in Moser--Trudinger inequality with first order derivative, we can not adapt Carleson--Chang's idea \cite{CC1986} of symmetrization to establish the existence of extremal function for inequalities of higher order derivatives. It is still a rather difficult problem to answer the above question in the most generality. One interesting case of the above question when $n=4$ and $m=2$ was addressed in \cite{LuYang}. Let $\Om \subset \R^4$ denote a smooth oriented bounded domain, $H_0^2(\Om)$ denote the Sobolev space which is completion of the space of compactly supported smooth functions in $\Om$ under the Dirichlet norm $\|u\|_{H_0^2(\Om)} = \|\Delta u\|_2$. Then Adams inequality in the case $n=4$ and $m=2$ states that
\begin{equation}\label{eq:Adams4dim}
\sup_{u\in H_0^2(\Om),\, \|\Delta u\|_2 \leq 1} \int_\Om e^{\gamma u^2} dx < \infty,
\end{equation}
for any $\gamma \leq 32 \pi^2$. The existence of extremal function for inequality \eqref{eq:Adams4dim} was proved by Lu and Yang \cite{LuYang}. Even, Lu and Yang established in \cite{LuYang} an improvement of \eqref{eq:Adams4dim} in spirit of Adimurthi and Druet (for improvement of Moser--Trudinger inequality \eqref{eq:AdiDruet}). Let $\lambda_1(\Om)$ denote the first eigenvalue of the bi-Laplacian operator $\Delta^2$ on $\Om$, i.e., 
\[
\lambda_1(\Om) = \inf_{u\in H_0^2(\Om),\, u\not\equiv 0} \, \frac{\|\Delta u\|_2^2}{\|u\|_2^2}.
\]
An easy application of the variational method shows that $\lambda_1(\Om) >0$ and is attained. It was proved by Lu and Yang that
\begin{equation}\label{eq:LuYang}
\sup_{u\in H_0^2(\Om),\, \|\Delta u\|_2 \leq 1} \int_\Om e^{32 \pi^2 q(\|u\|_2^2) u^2} dx < \infty,
\end{equation}
where $q(t) = 1 + a_1 t+ \cdots + a_k t^k$, $k\geq 1$ is a polynomial of order $k$ in $\R$ with $0\leq a_1 < \lambda_1(\Om)$, $0\leq a_2 \leq \lam_1(\Om) a_1$, \ldots, $0\leq a_k \leq \lam_1(\Om) a_{k-1}$. Furthermore, if $a_1\geq \lam_1(\Om)$ then the supremum above will be infinite.

The existence of extremal functions for inequality \eqref{eq:LuYang} was also studied in \cite{LuYang}. It was proved that there exists a strictly positive constant $\ep_0 < \lam_1(\Om)$ depending only on $\Om$ such that when $0\leq a_1 < \lambda_1(\Om)$, $0\leq a_2 \leq \lam_1(\Om) a_1$, \ldots, $0\leq a_k \leq \lam_1(\Om) a_{k-1}$, we can find $u^* \in H_0^2(\Om)\cap C^4(\o{\Om})$ such that $\|\Delta u^*\|_2 =1$ and 
\[
\int_\Om e^{32 \pi^2 q(\|u^*\|_2^2) {u^*}^2} dx = \sup_{u\in H_0^2(\Om),\, \|\Delta u\|_2 \leq 1} \int_\Om e^{32 \pi^2 q(\|u\|_2^2) u^2} dx.
\]
Obviously, this implies the existence of extremal functions for Adams inequality \eqref{eq:Adams4dim}.

The first aim of this paper is to strengthen Adams inequality \eqref{eq:Adams4dim} in the spirit of Tintarev for the improvement of Moser--Trudinger inequality \eqref{eq:Tintarev}.To do this, let us define for any $0\leq \alpha < \lam_1(\Om)$,
\[
\|u\|_{2,\alpha}^2 = \|\Delta u\|_2^2 -\alpha \|u\|_2^2,\qquad u\in H_0^2(\Om).
\]
Note that $\|\cdot\|_{2,\alpha}$ is a norm on $H^2_0(\Om)$ which is equivalent to $\|\cdot\|_{H_0^2(\Om)}$. In this paper, we will prove the following inequality.

\begin{theorem}\label{Main1}
Let $\Om \subset \R^4$ be a smooth, oriented bounded domain, $\lam_1(\Om)$ be the first eigenvalue of bi-Laplacian operator $\De^2$ on $\Om$. Then for any $\alpha$ with $0 \leq \alpha < \lam_1(\Om)$, we have
\begin{equation}\label{eq:Maininequality}
\sup_{u\in H_0^2(\Om), \, \|u\|_{2,\al} \leq 1} \int_{\Om} e^{32\pi^2 u^2} dx < \infty.
\end{equation}
\end{theorem}
Remark that when $\al =0$, \eqref{eq:Maininequality} reduces to \eqref{eq:Adams4dim}. Moreover, for any $0\leq \alpha < \lam_1(\Om)$ and $u \in H_0^2(\Om)$ such that $\|\Delta u\|_2 \leq 1$ denote $v = u/\|u\|_{2,\alpha}$, then $u^2 \leq v^2$, and $\|v\|_{2,\alpha} =1$, thus \eqref{eq:Maininequality} is indeed stronger than Adams inequality \eqref{eq:Adams4dim}. The next result shows that \eqref{eq:Maininequality} is stronger than the inequality of Lu and Yang \eqref{eq:LuYang}.

\begin{proposition}\label{imply}
Theorem \ref{Main1} implies the inequality \eqref{eq:LuYang}.
\end{proposition}

The second result of this paper is the existence of the extremal functions for the inequality \eqref{eq:Maininequality}. More precisely, we prove the following result.

\begin{theorem}\label{Existence}
Let $\Om \subset \R^4$ be a smooth, oriented bounded domain, $\lam_1(\Om)$ be the first eigenvalue of bi-Laplacian operator $\De^2$ on $\Om$. Then for any $\alpha$ with $0 \leq \alpha < \lam_1(\Om)$, there exists $u^* \in H_0^2(\Om) \cap C^4(\o{\Om})$ such that $\|u^*\|_{2,\al} =1$ and 
\[
\int_{\Om} e^{32\pi^2 {u^*}^2} dx = \sup_{u\in H_0^2(\Om), \, \|u\|_{2,\al} \leq 1} \int_{\Om} e^{32\pi^2 u^2} dx.
\]
\end{theorem}
Note that when $\alpha =0$ we obtain the existence of extremal function for Adams inequality \eqref{eq:Adams4dim} which was already proved in \cite{LuYang}. Although, our inequality \eqref{eq:Maininequality} is stronger than the one of Lu and Yang \eqref{eq:LuYang}, however the existence result in Theorem \ref{Existence} does not imply the existence result for the inequality \eqref{eq:LuYang}. Also, contrary with the existence result of Lu and Yang, our Theorem \ref{Existence} gives the existence of extremal function for the inequality \eqref{eq:Maininequality} for any $0\leq \al < \lam_1(\Om)$.

We conclude this introduction by mentioning about the method of proof of our main Theorems. As usually, our method is based on the blow-up analysis method. We first establish a concentration-compactness lemma of Lion's type and using it to prove the existence of $u_\ep \in H_0^2(\Om)\cap C^4(\o{\Om})$, $\ep \in (0,32\pi^2)$ such that $\|u_\ep\|_{2,\al} =1$ and
\[
\int_{\Om} e^{32\pi^2 {u_\ep}^2} dx = \sup_{u\in H_0^2(\Om), \, \|u\|_{2,\al} \leq 1} \int_{\Om} e^{(32\pi^2-\ep) u^2} dx.
\]
Thus, the Euler--Lagrange equation of $u_\ep$ is given by
\[
\begin{cases}
\Delta^2 u_\ep = \frac1{\lambda_\ep} e^{(32\pi^2 -\ep) u_\ep^2} u_\ep + \al u_\ep&\mbox{in $\Om$,}\\
\|u_\ep\|_{2,\alpha} =1, u_\ep = \frac{\pa u_\ep}{\pa \nu} = 0&\mbox{on $\pa \Om$,}\\
\lambda_\ep = \int_{\Om} e^{(32\pi^2 -\ep)u_\ep^2} u_\ep^2 dx,
\end{cases}
\]
where $\nu$ denotes the outward unit normal vector to $\pa \Om$. Without loss of generality, let $c_\ep = \max_{\o{\Om}} |u_\ep| = u_\ep(x_\ep)$. If $c_\ep$ is bounded, by the standard regularity theory we obtain $u_\ep \to u^*$ in $C^4(\o{\Om})$ hence finishes our proof. If $c_\ep \to \infty$ (namely, the blow-up occurs) and $x_\ep \to p\in \o{\Om}$, by using Pohozaev type identity and elliptic estimates, we exclude the case $p\in \pa \Om$. We also show that $c_\ep u_\ep$ converges to some Green function weakly in $H_0^2(\Om)$ which then immediately leads to Theorem \ref{Main1}. We also prove an upper bound for functional $\int_\Om e^{32\pi^2 u^2} dx$ when blow-up occurs by using some capacity estimates. By constructing a sequence of test functions, we exclude the blow-up phenomena for the maximizing sequence of functional $\int_\Om e^{32\pi^2 u^2} dx$. This leads to the existence result in Theorem \ref{Existence}. We emphasize here that in our proof below, we do not require the sharp Adams inequality (i.e., $\gamma =32\pi^2$ in \eqref{eq:Adams4dim}), but only require the subcritical Adams inequality (i.e., $\gamma < 32\pi^2$ in \eqref{eq:Adams4dim}). We also would like to mention here that blow-up analysis technique have been already employed by numerous authors in relevant but quite different setting in dealing with Sobolev inequalities instead of Moser--Trudinger inequality. We refer the interested reader to the works \cite{Aubin,Druet,LiZhu,doO2014,doOSouza,Li2001,Li2005,LiN2007,LiRuf2008,LuYang,WY2012,Yang2006a,Yang2006,Yang2007,Yang2015,Yang2017,Yang2017a}, etc.

The rest of this paper is organized as follows. In section \S2, we give the existence of maximizers for subcritical functional. In section \S3 we analyse the asymptotic behavior of those maximizers functions. In section \S4, we obtain an upper bound for the critical functional under the assumption that blow-up occurs in the interior of $\Om$ by using some capacity estimates. We exclude the boundary bubble in section \S5. The proof of Theorem \ref{Main1} and Proposition \ref{imply} is given in section \S6. In section \S7, we construct a sequence of test functions to conclude the existence of extremal function for the critical functional and thus give the proof of Theorem \ref{Existence}.


\section{Extremals for the subcritical Adams inequality}
For any $\ep \in (0,32\pi^2)$, let us consider the subcritical problems
\begin{equation}\label{eq:subcriticalproblem}
C_\ep = \sup_{u\in H_0^2(\Om), \, \|u\|_{2,\alpha} \leq 1} \int_{\Om} e^{(32\pi^2 -\ep) u^2} dx.
\end{equation}
In this section, we mainly prove that $C_\ep < \infty$ and the subcritical problem \eqref{eq:subcriticalproblem} is attained. Noting that the existence of such extremals is nontrivial. In the proof, we need the following Lion's type \cite{Lions1985} concentration--compactness principle.
\begin{proposition}\label{Lions}
Let $\{u_j\}_j\subset H_0^2(\Om)$ be a sequence of functions such that $\| u_j\|_{2,\alpha} =1$ and $u_j \rightharpoonup u_0$ weakly in $H_0^2(\Om)$. Then for any $p< (1-\|u_0\|_{2,\alpha}^2)^{-1}$,
\[
\limsup_{j\to\infty} \int_{\Om} e^{32\pi^2 p u_j^2} dx < \infty.
\]
\end{proposition}
\begin{proof}
By Rellich--Kondrachov theorem, we have $\|u_j\|_2 \to \|u_0\|_2$ as $j\to \infty$. Denote
\[
v_j = \frac{u_j}{\|\Delta u_j\|_2} = \frac{u_j}{(1 + \alpha \|u_j\|_2^2)^{1/2}},
\]
then $\|\Delta v_j\|_2 =1$ and $v_j \rightharpoonup v_0 = u_0/(1+ \al \|u_0\|_2)^{1/2}$ weakly in $H_0^2(\Om)$. Applying the Lions type concentration--compactness principle of Lu and Yang (see Proposition $3.1$ in \cite{LuYang}), we have
\begin{equation}\label{eq:LY}
\limsup_{j\to \infty} \int_\Om e^{32\pi^2 q v_j^2} dx < \infty
\end{equation}
for any $q < 1/(1 -\|\Delta v_0\|_2^2)$. For any $p < 1/(1 -\|u_0\|_{2,\al}^2)$ we have
\[
\lim_{j\to\infty} p \|\Delta u_j\|_2^2 = p(1 + \alpha\|u_0\|_2^2) < \frac{1+\alpha\|u_0\|_2^2}{1+ \alpha \|u_0\|_2^2 -\|\Delta u_0\|_2^2} =\frac1{1 -\|\Delta v_0\|_2^2}.
\]
This implies the existence of $j_0$ and $q < 1/(1-\|\Delta v_0\|_2^2)$ such that
\[
p \|\Delta u_j\|_2^2 \leq q < \frac1{1 -\|\Delta v_0\|_2^2},\quad\forall\, j\geq j_0.
\]
Thus by \eqref{eq:LY}, we get
\[
\limsup_{j\to\infty} \int_{\Om} e^{32\pi^2 p u_j^2} dx = \limsup_{j\to\infty} \int_\Om e^{32\pi^2 p\|\Delta u_j\|_2^2 v_j^2} dx \leq \limsup_{j\to\infty} \int_\Om e^{32\pi^2 q v_j^2} dx < \infty,
\]
as our desire.
\end{proof}

Our existence result is given in the following proposition.

\begin{proposition}\label{subcriticalexistence}
For any $\ep \in (0, 32\pi^2)$, we have $C_\ep < \infty$ and there exists $u_\ep \in H_0^2(\Om)$ such that $\|u_\ep\|_{2,\alpha} =1$ and 
\[
C_\ep = \int_{\Om} e^{(32\pi^2 -\ep) u_\ep^2} dx.
\]
Note that $32\pi^2 -\ep$ can be replaced by any sequence $\{\rho_\ep\}_\ep$ with $\rho_\ep \uparrow 32\pi^2$.
\end{proposition}
\begin{proof}
Let $\{u_j\}_j\subset H_0^2(\Om)$ be a sequence of functions with $\|u_j\|_{2,\alpha} =1$ and 
\[
\lim_{j\to \infty} \int_{\Om} e^{(32\pi^2 -\ep)u_j^2} dx = C_\ep.
\]
Since $\alpha \in [0, \lambda_1(\Om))$ then
\[
1 = \|u_j\|_{2,\alpha}^2 \geq \lt(1 -\frac{\alpha}{\lambda_1(\Om)}\rt) \|\Delta u_j\|_2^2.
\]
Thus $\{u_j\}_j$ is bounded in $H_0^2(\Om)$. Up to a subsequence, we can assume that $u_j\rightharpoonup u_\ep$ weakly in $H_0^2(\Om)$, $u_j \to u_\ep$ in $L^p(\Om)$ for any $1< p< \infty$ and $u_j \to u_\ep$ a.e., in $\Om$. If $u_\ep = 0$, then by Rellich--Kondrachov theorem, we have $\|u_j\|_2 \to 0$ as $j\to\infty$. Define 
\[
v_j = \frac{u_j}{\|\Delta u_j\|_2} = \frac{u_j}{(1 + \alpha \|u_j\|_2^2)^2},
\]
then $\|\Delta v_j\|_2 =1$ and $v_j \to 0$ a.e., in $\Om$. Since
\[
\lim_{j\to \infty} (32\pi^2 -\ep) (1 + \alpha \|u_j\|_2^2) = 32\pi^2 -\ep,
\]
and by Adams inequality
\[
\sup_{j\geq 1} \int_\Om e^{32\pi^2 v_j^2 } dx < \infty,
\]
then there exists $p >1$ such that
\[
\sup_{j} \int_{\Om} e^{(32\pi^2 -\ep) p u_j^2} dx < \infty.
\]
Thus, since $v_j\to 0$ a.e., in $\Om$, by letting $j\to\infty$ we get
\[
\lim_{j\to\infty} \int_{\Om} e^{(32\pi^2 -\ep) u_j^2} dx = |\Om|,
\]
which is impossible. Hence $u_\ep \not\equiv 0$ and
\[
0< \|\Delta u_\ep\|_2^2 -\alpha \|u_\ep\|_2^2 \leq \liminf_{j\to\infty} \|u_j\|_{2,\alpha}^2 \leq 1,
\]
It follows from Proposition \ref{Lions} that
\[
\sup_{j\geq 1} \int_\Om e^{32\pi^2 p u_j^2} dx < \infty
\]
for any $p < 1/(1-\|u_\ep\|_{2,\alpha}^2 )$. 
This together $u_j\to u_\ep$ a.e., in $\Om$ implies
\[
\lim_{j\to\infty} \int_\Om e^{(32\pi^2 -\ep) u_j^2} dx = \int_\Om e^{(32\pi^2 -\ep) u_\ep^2} dx.
\]
This shows that $C_\ep < \infty$. Obviously, we must have $\| u_\ep\|_{2,\alpha}=1$. Hence $u_\ep$ is a maximizer for $C_\ep$.
\end{proof}

An easy computation shows that the Euler--Lagrange equation of $u_\ep$ is given by
\begin{equation}\label{eq:ELequation}
\begin{cases}
\Delta^2 u_\ep = \frac1{\lambda_\ep} e^{\alpha_\ep u_\ep^2} u_\ep + \alpha u_\ep &\mbox{in $\Om$}\\
\|u_\ep\|_{2,\alpha} = 1,\quad u_\ep = \frac{\partial u_\ep}{\partial \nu} =0&\mbox{on $\pa \Om$}\\
\alpha_\ep = 32\pi^2 -\ep, \quad \lambda_\ep = \int_\Om e^{\alpha_\ep u_\ep^2} u_\ep^2 dx.
\end{cases}
\end{equation}

\begin{lemma}\label{canduoilambdaepsilon}
It holds $\liminf_{\ep \to 0} \lambda_\ep  >0$.
\end{lemma}
\begin{proof}
Using the inequality $e^t \leq 1 + te^t$ for $t\geq 0$, we get
\begin{equation}\label{eq:estimate1}
C_\ep = \int_\Om e^{\alpha_\ep u_\ep^2} dx \leq |\Om| + \alpha_\ep \lambda_\ep.
\end{equation}
It is evident that
\[
\limsup_{\ep\to 0} C_\ep \leq \sup_{u\in H_0^2(\Om),\, \|u\|_{2,\alpha} =1} \int_\Om e^{32\pi^2 u^2} dx.
\]
For any $u\in H_0^2(\Om)$ with $\|u\|_{2,\alpha} =1$, by Fatou's lemma we have
\[
\int_\Om e^{32\pi^2 u^2} dx \leq \liminf_{\ep\to 0} \int_\Om e^{\alpha_\ep u^2} dx \leq \liminf_{\ep \to 0} C_\ep.
\]
Taking the supremum over all such functions $u$, we have
\[
\sup_{u\in H_0^2(\Om),\, \|u\|_{2,\alpha} =1} \int_\Om e^{32\pi^2 u^2} dx \leq \liminf_{\ep \to 0} C_\ep.
\]
Thus we have shown that
\begin{equation}\label{eq:estimate2}
\lim_{\ep \to 0} C_\ep =\sup_{u\in H_0^2(\Om),\, \|u\|_{2,\alpha} =1} \int_\Om e^{32\pi^2 u^2} dx > |\Om|.
\end{equation}
Combining \eqref{eq:estimate1} and \eqref{eq:estimate2} together we obtain the desired estimate.
\end{proof}

\section{Asymptotic behavior of extremals for subcritical functionals}
The crucial tool in studying the regularity of higher order equations is the Green's representation formula. The Green function $G(x,y)$ for $\Delta^2$ under the Dirichlet condition is the solution of
\begin{equation}\label{eq:Greenfunction}
\Delta^2 G(x,y) = \de_x(y)\quad\text{\rm in }\Om,\qquad
G(x,y) = \frac{\partial G(x,y)}{\partial \nu} = 0\quad \text{\rm on } \pa \Om.
\end{equation}
All functions $u\in H_0^2(\Om) \cap C^4(\o{\Om})$ satisfying $\Delta^2 u = f$ can be represented by
\[
u(x) = \int_{\Om} G(x,y) f(y) dy.
\]
We will need the following useful estimates \cite{DS2004} for $G$ in the analysis below
\begin{equation}\label{eq:Greenestimate}
|G(x,y)| \leq C \ln(2 + |x-y|^{-1}),\qquad |\na^iG(x,y)|\leq C|x-y|^{-i},\quad i\geq 1,
\end{equation}
for some constant $C >0$ and for all $x,y\in \Om$, $x\not=y$.

Denote $c_\ep = \max_{x\in \Om} |u_\ep(x)| = |u_\ep (x_\ep)|$ for $x_\ep \in \Om$. If $c_\ep$ is bounded, then applying the standard regularity to \eqref{eq:ELequation} we obtain $u_\ep \to u^*$ in $C^4(\o{\Om})$ for some $u^* \in H_0^2(\Om) \cap C^4(\o{\Om})$ with $\|u^*\|_{2,\alpha} =1$. This then implies 
\[
\int_{\Om} e^{32\pi^2 {u^*}^2} dx = \sup_{u\in H_0^2(\Om),\, \|u\|_{2,\al} =1} \int_{\Om} e^{32\pi^2 u^2} dx,
\]
which leads to our desired results.

In the sequel, we assume that $c_\ep \to \infty$. Without loss of generality we assume that
\begin{equation}\label{eq:assumptiononblowupsequence}
c_\ep = u_\ep(x_\ep) = \max_{x\in \Om} |u_\ep(x)| \to \infty,\quad x_\ep \in \Om,\quad x_\ep \to p \in \o{\Om}\quad\text{\rm as}\quad \ep \to 0.
\end{equation}
As in \cite{LuYang}, we call $p$ the blow-up point. Here and in the sequel, we do not distinguish sequence and subsequence, the reader can understand it from the context.

Since $\|u_\ep\|_{2,\alpha} =1$ and $\alpha < \lambda_1(\Om)$ then
\[
\|\Delta u_\ep\|_2^2 \leq 1 + \frac{\alpha}{\lambda_1(\Om)},
\]
hence $u_\ep$ is bounded in $H_0^2(\Om)$, we can assume that $u_\ep \rightharpoonup u_0$ weakly in $H_0^2(\Om)$, $u_\ep \to u_0$ in $L^s(\Om)$ for any $1< s <\infty$ and $u_\ep \to u_0$ a.e., in $\Om$. If $u_0\not\equiv 0$, then by Lions type concentration--compactness principle (Proposition \ref{Lions}), there is $p >1$ such that
\[
\sup_{\ep >0} \int_{\Om} e^{32\pi^2 p u_\ep^2} dx < \infty.
\]
Hence $e^{\alpha_\ep u_\ep}$ is bounded in $L^r(\Om)$ for some $r > 1$ provided that $\ep$ is small enough. Applying the standard regularity theory to \eqref{eq:ELequation}, we obtain the boundedness of $c_\ep$ which is contradiction with \eqref{eq:assumptiononblowupsequence}. Hence, we have
\begin{equation}\label{eq:limitofblowupsequence}
\begin{cases}
u_\ep \rightharpoonup 0&\mbox{weakly in $H_0^2(\Om)$,}\\
u_\ep \to 0&\mbox{in $L^r(\Om)$ for any $r >1$, and a.e., in $\Om$,}\\
\alpha_\ep \to 32\pi^2.
\end{cases}
\end{equation}

In the rest of this section we focus on the case $p\in \Om$ (the case $p\in \pa \Om$ will be treated below in \S5). We claim that
\begin{equation}\label{eq:claim}
|\Delta u_\ep|^2 dx \rightharpoonup \de_p\qquad\text{\rm in the sense of measure.}
\end{equation}
Indeed, if \eqref{eq:claim} does not hold. Since $\|\Delta u_\ep\|_2^2 = 1 + \alpha \|u_\ep\|_2^2 \to 1$ as $\ep \to 0$, we can find $r >0$ and $\eta >0$ such that $B_r(p) \subset \Om$ and
\[
\limsup_{\ep\to 0} \int_{B_r(p)} |\Delta u_\ep|^2 dx \leq 1 -\eta.
\]
From Sobolev embedding theorem and \eqref{eq:limitofblowupsequence}, we have $\na u_\ep \to 0$ strongly in $L^2(\Om)$. Let $\phi\in C_0^\infty(B_r(p))$ be a cut-off function with $0\leq \phi\leq 1$ and $\phi=1$ on $B_{r/2}(p)$. We have
\[
\limsup_{\ep\to 0} \int_{B_r(p)} |\Delta (\phi u_\ep)|^2 dx \leq 1 -\eta.
\]
By Adams inequality, $e^{\alpha_\ep \phi^2 u_\ep^2}$ is bounded in $L^{2/(2-\eta)}(\Om)$ and hence $e^{\alpha_\ep u_\ep^2}$ is bounded in $L^{2/(2-\eta)}(B_{r/2}(p))$ provided that $\ep$ is small enough. Applying the standard regularity theory to \eqref{eq:ELequation}, we have that $u_\ep$ is bounded in $C^1(\o{B_{r/4}(p)})$. This contradicts our assumption \eqref{eq:assumptiononblowupsequence}. Hence, we obtain \eqref{eq:claim}. In fact, we have shown that there is no other blow-up point if $p$ lies in $\Om$ and $\|u_\ep\|_{2,\alpha} =1$.

To proceed, we introduce the following quantities
\begin{equation}\label{eq:quantities}
b_\ep = \frac{\lambda_\ep}{\int_\Om |u_\ep| e^{\alpha_\ep u_\ep^2} dx},\quad \tau = \lim_{\ep \to 0} \frac{c_\ep}{b_\ep},\quad \sigma = \lim_{\ep\to 0} \frac{\int_\Om u_\ep e^{\alpha_\ep u_\ep^2} dx}{\int_\Om |u_\ep| e^{\alpha_\ep u_\ep^2} dx}.
\end{equation}
Note that $\tau \geq 1$ or $\tau = \infty$, $|\sigma| \leq 1$. We will show that $\sigma =1$ at the end of this section.

Let
\[
r_\ep^4 = \frac{\lam_\ep}{c_\ep^2} e^{-\al_\ep c_\ep^2}, \quad \Om_\ep =\{x\in \R^4\, :\, x_\ep + r_\ep x \in \Om\}.
\]
We will show that $r_\ep$ converges to zero rapidly. Indeed, for any $0< \gamma < 32\pi^2$, we have
\begin{equation}\label{eq:convergezero}
r_\ep^4 c_\ep^2 e^{\gamma c_\ep^2} = e^{(\gamma-\alpha_\ep) c_\ep^2}\int_{\Om} u_\ep^2 e^{\alpha_\ep u_\ep^2} dx \leq \int_{\Om} u_\ep^2 e^{\gamma u_\ep^2} dx \to 0,
\end{equation}
here we used H\"older inequality, \eqref{eq:limitofblowupsequence} and the fact $0< \gamma < 32\pi^2$. In particular, $r_\ep \to 0$ and $\Om_\ep \to \R^4$ as $\ep\to 0$. We next define two sequences of functions on $\Om_\ep$ by
\[
\psi_\ep(x) = \frac{u_\ep(x_\ep + r_\ep x)}{c_\ep},\quad \varphi_\ep(x) = b_\ep(u_\ep(x_\ep + r_\ep x) -c_\ep) = b_\ep c_\ep(\psi_\ep(x) -1).
\]
Our next goal is to understand the asymptotic behavior of $\psi_\ep$ and $\varphi_\ep$. Evidently, $|\psi_\ep| \leq 1$ and 
\[
\Delta^2 \psi_\ep(x) = r_\ep^4\lt(\frac1{\lambda_\ep} \psi_\ep(x) e^{\alpha_\ep u_\ep(x_\ep + r_\ep x)^2} + \alpha \psi_\ep(x)\rt).
\]
Thus, for any $R >0$ and $x\in B_R(0)$ we have
\[
|\Delta^2 u_\ep(x)|^2 \leq \frac1{c_\ep^2} + \alpha r_\ep^4 \to 0,
\]
and
\[
\int_{B_R(0)}|\Delta \psi_\ep|^2 dx = \frac1{c_\ep^2}\int_{B_{r_\ep R}(x_\ep)} |\Delta u_\ep|^2 dx \to 0.
\]
These estimates and the standard regularity theory give $\psi_\ep \to \psi$ in $C^4_{\rm loc}(\R^4)$ with $\Delta \psi =0$ in $\R^4$. Note that $|\psi_\ep|\leq 1$ and $\psi_\ep(0) =1$, then $|\psi|\leq 1$ and $\psi(0) =1$. Using Liouville theorem, we conclude that $\psi \equiv 1$ in $\R^4$. Thus, we have proved that
\begin{lemma}\label{limitofpsiepsilon}
It holds $\psi_\ep \to 1$ in $C^4_{\rm loc}(\R^4)$.
\end{lemma}

We next investigate the convergence of $\varphi_\ep$.
\begin{lemma}\label{limitofvarphiepsilon}
Let $\tau$ be defined in \eqref{eq:quantities}. Then $\varphi_\ep \to \varphi$ in $C^4_{\rm loc}(\R^4)$, where
\begin{equation}\label{eq:varphifunction}
\varphi(x) = 
\begin{cases}
\frac1{16\pi^2\tau} \ln\, \frac1{1 + \frac{\pi}{\sqrt{6}}|x|^2} &\mbox{if $\tau < \infty$,}\\
0 &\mbox{if $\tau = \infty$,}
\end{cases}
\end{equation}
for $x\in \R^4$.
\end{lemma}
\begin{proof}
Using Green representation formula, we have
\[
u_\ep(x) = \int_\Om G(x,y) \lt(\frac1{\lambda_\ep} e^{\alpha_\ep u_\ep(y)^2} u_\ep(y) + \alpha u_\ep(y)\rt) dy
\]
hence 
\[
\nabla^i u_\ep(x) = \int_\Om \nabla^i_x G(x,y)\lt(\frac1{\lambda_\ep} e^{\alpha_\ep u_\ep(y)^2} u_\ep(y) + \alpha u_\ep(y)\rt) dy,
\]
for $i =1,2$. Thus, for any $R >0$, $x\in B_R(0)$ and $i =1,2$, by using \eqref{eq:Greenestimate} we have
\begin{align}\label{eq:oxox}
|\na^i \varphi_\ep(x)| & = \lt|r_\ep^i b_\ep \int_\Om \na^i_x G(x_\ep + r_\ep x,y)\lt(\frac1{\lambda_\ep} e^{\alpha_\ep u_\ep(y)^2} u_\ep(y) + \alpha u_\ep(y)\rt) dy \rt|\notag\\
&\leq C b_\ep r_\ep^i \int_\Om \lt(\frac1{\lambda_\ep} \frac{|u_\ep(y)|e^{\alpha_\ep u_\ep(y)^2}}{|x_\ep + r_\ep x -y|^i} + \frac{\alpha |u_\ep(y)|}{|x_\ep + r_\ep x -y|^i}\rt) dy\notag\\
&\leq C b_\ep r_\ep^i \Bigg(\int_{B_{2R r_\ep}(x_\ep)}\frac1{\lambda_\ep} \frac{|u_\ep(y)|e^{\alpha_\ep u_\ep(y)^2}}{|x_\ep + r_\ep x -y|^i} dy + \int_{\Om\setminus B_{2Rr_\ep}(x_\ep)} \frac1{\lambda_\ep} \frac{|u_\ep(y)|e^{\alpha_\ep u_\ep(y)^2}}{|x_\ep + r_\ep x -y|^i} dy\notag\\
&\qquad\qquad\qquad\qquad + \int_\Om \frac{\alpha |u_\ep(y)|}{|x_\ep + r_\ep x -y|^i} dy\Bigg)\notag\\
&\leq C\Bigg(\frac{b_\ep}{c_\ep} \int_{B_{2R}(0)} \frac{dz}{|x-z|^i}+ \frac1{R^i} +\alpha b_\ep r_\ep^i c_\ep \int_\Om \frac{dy}{|x_\ep + r_\ep x -y|^i} \Bigg)\notag\\
&\leq C(R),
\end{align}
here we use \eqref{eq:convergezero} and $b_\ep \leq c_\ep$.

A straightforward computation shows that $\varphi_\ep$ satisfies
\begin{equation}\label{eq:varphiepsiloneq}
\Delta^2 \varphi_\ep(x) = \frac{b_\ep}{c_\ep} \psi_\ep(x) e^{\alpha_\ep \frac{c_\ep}{b_\ep} (1+ \psi_\ep(x)) \varphi_\ep(x)} + \alpha b_\ep c_\ep r_\ep^4 \psi_\ep(x).
\end{equation}
Since $b_\ep \leq c_\ep$, $\psi_\ep \to 1$ in $C^4_{\rm loc}(\R^4)$, \eqref{eq:convergezero}, $\varphi_\ep \leq 0$ and \eqref{eq:oxox}, we obtain by applying the standard regularity theorey to \eqref{eq:varphiepsiloneq} that $\varphi_\ep \to \varphi$ in $C^4_{\rm loc}(\R^4)$ for some function $\varphi$. We have two following cases.\\

\emph{$\bullet$ Case 1: $\tau < \infty$.} By letting $\ep \to 0$, then using Lemma \ref{limitofpsiepsilon}, \eqref{eq:convergezero} and \eqref{eq:varphiepsiloneq} we obtain
\begin{equation}\label{eq:varphicondition}
\Delta^2 \varphi(x) = \frac1{\tau} e^{64 \pi^2\tau \varphi(x)} ,\qquad \varphi(x) \leq \varphi(0) =0,\qquad \int_{\R^4} e^{64\pi^2 \tau \varphi(x)} dx < \infty.
\end{equation}
Indeed, for any $R >0$, we have $\psi_\ep(x) = 1 + o_{\ep,R}(1)$ where $o_{\ep,R}(1)$ means that 
\[
\lim_{\ep\to 0} o_{\ep,R}(1) =0\quad\text{\rm uniformly in } B_R(0).
\]
Thus $\psi_\ep^2(x) =1 + o_{\ep,R}(1)$ for $x\in B_R(0)$ or equivalently $u_\ep(x_\ep + r_\ep x)^2 = c_\ep^2(1+ o_{\ep,R}(1))$ for $x\in B_R(0)$. Hence
\[
\lambda_\ep = \int_{\Om} u_\ep^2 e^{\alpha_\ep u_\ep^2} dx \geq c_\ep^2(1+ o_{\ep,R}(1)) \int_{B_{R r_\ep}(x_\ep)} e^{\alpha_\ep u_\ep^2} dx.
\]
Applying Fatou's lemma, we have
\begin{align*}
\int_{B_R(0)} e^{64\pi^2\tau \varphi^2(x)} dx & \leq \liminf_{\ep\to 0} \int_{B_R(0)} e^{\alpha_\ep \frac{c_\ep}{b_\ep} (1+ \psi_\ep(x)) \varphi_\ep(x)} dx\\
&=\liminf_{\ep\to 0} \int_{B_R(0)} e^{\alpha_\ep(u_\ep(x_\ep +r_\ep x)^2 -c_\ep^2)} dx\\
&=\liminf_{\ep \to 0} r_\ep^{-4} \int_{B_{Rr_\ep}(x_\ep)} e^{\alpha_\ep(u_\ep(x)^2 -c_\ep^2)} dx\\
&=\liminf_{\ep\to 0} \frac{c_\ep^2 \int_{B_{Rr_\ep}(x_\ep)} e^{\alpha_\ep u_\ep(x)^2} dx}{\lambda_\ep}\\
&\leq \liminf_{\ep\to 0} \frac{c_\ep^2 \int_{B_{Rr_\ep}(x_\ep)} e^{\alpha_\ep u_\ep(x)^2} dx}{c_\ep^2(1+ o_{\ep,R}(1))\int_{B_{Rr_\ep}(x_\ep)} e^{\alpha_\ep u_\ep(x)^2} dx}\\
&=1,
\end{align*}
for any $R >0$. Letting $R\to \infty$ we get $\int_{\R^4} e^{64 \pi^2 \tau \varphi} dx < \infty$.

Moreover, we have
\[
\Delta\varphi_\ep(x) = b_\ep r_\ep^2 \int_\Om \Delta_x G(x_\ep +r_\ep x,y) \lt(\frac1{\lambda_\ep} e^{\alpha_\ep u_\ep(y)^2} u_\ep(y) + \alpha u_\ep(y)\rt) dy.
\]
Hence, for any $R >0$, by Fubini theorem we get
\begin{align*}
\int_{B_R(0)} |\Delta \varphi_\ep(x)| dx &\leq C b_\ep r_\ep^2 \int_\Om\frac1{\lambda_\ep} e^{\alpha_\ep u_\ep(y)^2} |u_\ep(y)| \int_{B_R(0)} \frac1{|x_\ep +r_\ep x -y|^2} dx dy\\
&\qquad + C\alpha b_\ep r_\ep^2 \int_\Om |u_\ep(y)| \int_{B_R(0)} \frac1{|x_\ep +r_\ep x -y|^2} dx dy\\
&\leq C' R^2,
\end{align*}
with $C'$ independent of $R$ and $\ep$. Letting $\ep \to 0$, we obtain
\[
\int_{B_R(0)} |\Delta \varphi(x)| dx \leq C' R^2,
\]
for any $R >0$ with $C'$ independent of $R$. This fact together \eqref{eq:varphicondition} and the results in \cite{CSL1998,WeiXu1999} implies that
\[
\varphi(x) = \frac1{16\pi^2\tau} \ln\, \frac1{1 + \frac{\pi}{\sqrt{6}}|x|^2},\qquad x\in \R^4.
\]

\emph{$\bullet$ Case 2: $\tau =\infty$.} From \eqref{eq:oxox} we obtain by letting $\ep \to 0$ that
\[
|\Delta \varphi(x)|\leq \frac{C}{R^2},
\]
for any $x\in B_R(0)$ and for any $R >0$ with $C$ independent of $R$. Let $R\to\infty$ we get $\Delta\varphi(x) =0$ for any $x\in \R^4$. Since $\varphi(x) \leq \varphi(0) =0$ for any $x\in \R^4$, then by Liouville Theorem, we conclude that $\varphi \equiv 0$.
\end{proof}

We next consider the asymptotic behavior of $u_\ep$ away from the blow-up point $p$. We have the following result.

\begin{lemma}\label{awayp}
$b_\ep u_\ep$ is bounded in $H_0^{2,r}(\Om)$ for any $1< r < 2$. In particular, there exists a constant $C$ depending only on $\Om$, $\lambda_1(\Om)$ and $\alpha_0$ such that $\|b_\ep u_\ep\|_{H_0^{2,r}(\Om)}\leq C$ uniformly for $\alpha\in [0,\alpha_0]$ with $\alpha_0 < \lambda_1(\Om)$.
\end{lemma}
\begin{proof}
Let $v_\ep$ be the solution of
\[
\begin{cases}
\Delta^2 v_\ep = \frac{1}{\lambda_\ep} b_\ep u_\ep e^{\alpha_\ep u_\ep^2} &\mbox{in $\Om$,}\\
v_\ep = \frac{\partial v_\ep}{\partial \nu} =0 &\mbox{on $\pa \Om$.}
\end{cases}
\]
By Green representation formula, we have
\[
v_\ep(x) = \int_\Om G(x,y)\frac{1}{\lambda_\ep} b_\ep u_\ep(y) e^{\alpha_\ep u_\ep(y)^2} dy
\]
and hence for any $i =1,2$, it holds
\[
|\na^iv_\ep(x)|\leq C \frac{b_\ep}{\lambda_\ep} \int_\Om |x-y|^{-i} |u_\ep (y)|e^{\alpha_\ep u_\ep(y)^2} dy = C\int_\Om |x-y|^{-i} \frac{|u_\ep(y)| e^{\alpha_\ep u_\ep(y)^2}}{\int_\Om |u_\ep(z)| e^{\alpha_\ep u_\ep(z)^2}dz} dy.
\]
Applying H\"older inequality, we obtain for any $1<r< 2$ that
\[
|\na^i v_\ep(x)|^r \leq C^r\int_\Om |x-y|^{-ir} \frac{|u_\ep(y)| e^{\alpha_\ep u_\ep(y)^2}}{\int_\Om |u_\ep(z)| e^{\alpha_\ep u_\ep(z)^2}dz} dy.
\]
Thus, by Fubini theorem, we have $\|\na^i v_\ep\|_r\leq C$ for $i=1,2$, whence
\begin{equation}\label{eq:bound1}
\|v_\ep\|_{H_0^{2,r}} \leq C.
\end{equation}
Let $w_\ep = b_\ep u_\ep -v_\ep$, then $w_\ep$ satisfies
\[
\begin{cases}
\Delta^2 w_\ep = \alpha w_\ep + \alpha v_\ep&\mbox{in $\Om$,}\\
w_\ep = \frac{\partial w_\ep}{\partial \nu} =0,&\mbox{on $\pa \Om$.}
\end{cases}
\]
Using $w_\ep$ as testing function for this equation, we get
\[
\|\Delta w_\ep\|_2^2 = \alpha \|w_\ep\|_2^2 + \alpha\int_\Om v_\ep w_\ep dx \leq \frac{\alpha}{\lambda_1(\Om)} \|\Delta w_\ep\|_2^2 + \frac{\alpha}{\sqrt{\lambda_1(\Om)}} \|v_\ep\|_2 \|\Delta w_\ep\|_2.
\]
Thus
\[
\lt(1 -\frac{\alpha}{\lambda_1(\Om)}\rt) \|\Delta w_\ep\|_2 \leq \frac{\alpha}{\sqrt{\lambda_1(\Om)}} \|v_\ep\|_2,
\]
which together \eqref{eq:bound1} and Sobolev inequality yields $\|\Delta w_\ep\|_2 \leq C$ with $C$ depends on $\Om$, $\lambda_1(\Om)$, and $\alpha_0 < \lambda_1(\Om)$ such that $0\leq \alpha \leq \alpha_0$. Hence $\|w_\ep\|_{H_0^2(\Om)} \leq C$ which together \eqref{eq:bound1} implies that $b_\ep u_\ep$ is bounded in $H_0^{2,r}(\Om)$ for any $1< r < 2$.
\end{proof}

We proceed by showing that $b_\ep u_\ep$ converges to some Green function.

\begin{lemma}\label{Greenfunctionlimit}
It holds $b_\ep u_\ep \rightharpoonup G_\alpha(\cdot,p)$ in $H_0^{2,r}(\Om)$ for any $1< r < 2$ with
\begin{equation}\label{eq:Greenfunctionlimit}
\begin{cases}
\Delta^2 G_\alpha(\cdot,p) = \sigma \de_p + \alpha G_\alpha(\cdot,p)&\mbox{in $\Om$,}\\
G_\alpha(\cdot,p) = \frac{\pa G_\alpha(\cdot,p)}{\pa \nu} =0&\mbox{on $\pa \Om$.}
\end{cases}
\end{equation}
Furthermore, $b_\ep u_\ep \to G_\al(\cdot,p)$ in $C^4_{\rm loc} (\o{\Om} \setminus \{p\})$. Also, we have
\begin{equation}\label{eq:decompositionGreen}
G_\alpha(x,p) = -\frac{\sigma}{8\pi^2} \ln \, |x-p| + A_p + \psi(x),
\end{equation}
where $A_p$ is constant depending on $p$ and $\alpha$, $\psi \in C^3(\o{\Om})$ and $\psi(p) =0$.
\end{lemma}
\begin{proof}
By Lemma \ref{awayp}, there exists a function $G_\al(\cdot,p) \in H_0^{2,s}(\Om)$ such that $b_\ep u_\ep \rightharpoonup  G_\alpha(\cdot,p)$ weakly in $H_0^{2,s}(\Om)$ for any $1< s < 2$. For any $r >0$ such that $B_r(p) \subset \Om$, by \eqref{eq:claim} we have $e^{\alpha_\ep u_\ep^2}$ is bounded in $L^s(\Om\setminus B_r(p))$ for any  $s >1$ (this is based on Adams inequality and cut-off function argument). Hence, by the standard regularity theory we obtain $b_\ep u_\ep \to G_\alpha(\cdot,p)$ in $C^4_{\rm loc}(\o{\Om}\setminus \{p\})$. Notice that $b_\ep u_\ep$ satisfies
\begin{equation}\label{eq:beue}
\begin{cases}
\Delta^2(b_\ep u_\ep) = \frac{1}{\lambda_\ep} b_\ep u_\ep e^{\al_\ep u_\ep^2} + \alpha b_\ep u_\ep&\mbox{in $\Om$,}\\
b_\ep u_\ep = \frac{\pa (b_\ep u_\ep)}{\pa \nu} = 0&\mbox{on $\pa \Om$.}
\end{cases}
\end{equation}
For any $\phi \in C^\infty(\o{\Om})$ we have
\begin{align}\label{eq:lim*}
\int_\Om \phi\, \lt(\frac{1}{\lambda_\ep} b_\ep u_\ep e^{\al_\ep u_\ep^2} + \alpha b_\ep u_\ep\rt) dx& =\int_{\Om} (\phi -\phi(p))\frac{1}{\lambda_\ep} b_\ep u_\ep e^{\al_\ep u_\ep^2} dx\notag\\
&\quad + \phi(p) \int_\Om \frac{1}{\lambda_\ep} b_\ep u_\ep e^{\al_\ep u_\ep^2}  dx + \alpha \int_\Om b_\ep u_\ep \phi dx.
\end{align}
Note that
\begin{equation}\label{eq:limit1}
\lim_{\ep\to 0} \int_\Om b_\ep u_\ep \phi dx = \int_\Om G_\alpha(x, p) \phi(x) dx,
\end{equation}
and
\begin{equation}\label{eq:limit2}
\lim_{\ep\to 0} \int_\Om \frac{1}{\lambda_\ep} b_\ep u_\ep e^{\al_\ep u_\ep^2}  dx = \lim_{\ep \to 0} \frac{\int_\Om u_\ep(x) e^{\alpha_\ep u_\ep(x)^2} dx}{\int_\Om |u_\ep(x)| e^{\alpha_\ep u_\ep(x)^2} dx} = \sigma.
\end{equation}
We will show that
\begin{equation}\label{eq:limitzero}
\lim_{\ep \to 0} \int_{\Om} (\phi -\phi(p))\frac{1}{\lambda_\ep} b_\ep u_\ep e^{\al_\ep u_\ep^2} dx=0.
\end{equation}
Indeed, by Lebesgue dominated convergence theorem, we get that
\[
\lim_{\ep\to 0} \int_{\{|u_\ep|\leq 1\}} e^{\alpha_\ep u_\ep^2} dx = |\Om|.
\]
This limit and \eqref{eq:estimate2} imply
\begin{align*}
\liminf_{\ep\to 0} \int_{\Om}|u_\ep| e^{\alpha_\ep u_\ep^2} dx &\geq \liminf_{\ep\to 0} \int_{\{|u_\ep|\geq 1} e^{\alpha_\ep u_\ep^2} dx\\
&=\liminf_{\ep\to 0} \lt(C_\ep -\int_{\{|u_\ep|\leq 1\}} e^{\alpha_\ep u_\ep^2} dx\rt)\\
&= \sup_{u\in H_0^2(\Om),\, \|u\|_{2,\alpha} =1} \int_\Om e^{32\pi^2 u^2} dx -|\Om|\\
&>0,
\end{align*}
hence $b_\ep/\lambda_\ep$ is bounded. For any $r >0$ with $B_r(p)\subset \Om$, we know that $e^{\alpha_\ep u_\ep^2}$ is bounded in $L^s(\Om\setminus B_r(p))$ for some $s >1$ and $u_\ep \to 0$ in $L^t(\Om)$ for any $t >1$, hence
\begin{equation}\label{eq:x1}
\lim_{\ep\to 0} \int_{\Om\setminus B_r(p)} (\phi -\phi(p)) \frac{b_\ep}{\lambda_\ep} u_\ep e^{\alpha_\ep u_\ep^2}dx =0.
\end{equation}
In the other hand
\[
\lt|\int_{B_r(p)} (\phi -\phi(p)) \frac{b_\ep}{\lambda_\ep} u_\ep e^{\alpha_\ep u_\ep^2} dx \rt|\leq \sup_{x\in B_r(p)} |\phi(x) -\phi(p)| \frac{b_\ep}{\lambda_\ep} \int_{\Om} |u_\ep| e^{\alpha_\ep u_\ep^2} dx = \sup_{x\in B_r(p)} |\phi(x) -\phi(p)|.
\]
Thus 
\begin{equation}\label{eq:x2}
\lim_{r\to\infty} \lim_{\ep\to 0} \int_{B_r(p)} (\phi -\phi(p)) \frac{b_\ep}{\lambda_\ep} u_\ep e^{\alpha_\ep u_\ep^2} dx = 0.
\end{equation}
\eqref{eq:limitzero} follows from \eqref{eq:x1} and \eqref{eq:x2}.

Plugging \eqref{eq:limit1}, \eqref{eq:limit2} and \eqref{eq:limitzero} into \eqref{eq:lim*} we obtain
\[
\lim_{\ep\to 0} \int_\Om \phi\, \lt(\frac{1}{\lambda_\ep} b_\ep u_\ep e^{\al_\ep u_\ep^2} + \alpha b_\ep u_\ep\rt) dx = \sigma \phi(p) + \alpha \int_\Om G_\alpha(x,p) \phi(x) dx,
\]
for any $\phi \in C^\infty(\o{\Om})$, hence
\[
\Delta^2 G_\alpha(\cdot,p) = \sigma \de_p + \alpha G_\alpha(\cdot, p)\quad\text{\rm in } \Om.
\]

The last conclusion was proved in the proof of Lemma $4.4$ in \cite{LuYang}. Let is recall it here for convenience of reader. Fix $r >0$ such that $B_{2r}(p) \subset \Om$ and consider the cut-off function $\phi \in C_0^\infty(B_{2r}(p)$ such that $\phi \equiv 1$ in $B_r(p)$. Let 
\[
g(x) = G_\alpha(x,p) + \frac{\sigma}{8\pi^2} \eta(x) \ln |x-p|.
\]
Then we have
\[
\begin{cases}
\Delta^2 g = f &\mbox{in $\Om$,}\\
g = \frac{\pa g}{\pa \nu} = 0&\mbox{on $\pa \Om$,}
\end{cases}
\]
with 
\begin{multline*}
f(x) =-\frac{\sigma^2}{8\pi^2} \Bigg(\Delta^2 \phi(x)\, \ln |x-p| + 2 \na \Delta \phi(x) \, \na \ln |x-p| + 2\Delta \phi(x) \, \Delta \ln |x-p| \\
+ 2\Delta(\na \phi(x)\, \na \ln |x-p|) + 2\na \phi(x)\, \na \Delta \ln |x-p|\Bigg) + \alpha G_\alpha(x,p).
\end{multline*}
Lemma \ref{awayp} and Sobolev inequality implies that $f \in L^s(\Om)$ for any $s>1$. By the standard regularity theory, we have $g \in C^3(\o{\Om})$. Let $A_p =g(p)$ and
\[
\psi(x) = g(x) -g(p) + \frac{\sigma}{8\pi^2} (1-\phi(x)) \ln |x-p|,
\]
we obtain the desired result.
\end{proof}

We continue by using Pohozaev type identity to find an upper bound of $\int_\Om e^{\alpha_\ep u_\ep^2}dx$. The following Pohozaev type identity is very useful in our analysis below.

\begin{lemma}\label{Pohozaev}
Assume $\Om'\subset \R^4$ is a smooth bounded domain. Let $u\in C^4(\o{\Om'})$ be a solution of $\Delta^2 u = f(u)$ in $\Om'$. Then we have for any $y\in \R^4$
\begin{align*}
4\int_{\Om'} F(u) dx &= \int_{\pa \Om'} \la x-y, \nu\ra F(u) d\om + \frac12 \int_{\pa \Om'} v^2 \la x-y,\nu\ra d\om + 2 \int_{\pa \Om'} \frac{\pa u}{\pa \nu} v d\om\\
&\quad + \int_{\pa \Om'} \lt(\frac{\pa v}{\pa \nu} \la x-y, \na u\ra + \frac{\pa u}{\pa \nu}\la x-y, \na v\ra -\la \na u,\na v\ra \la x-y,\nu\ra\rt) d\om,
\end{align*}
where $F(u) = \int_0^u f(s) ds$, $v =-\De u$ and $\nu$ is the normal outward derivative of $x$ on $\pa \Om'$.
\end{lemma}
The proof of this Pohozaev type identity can be found in \cite{EM1993,Robert}. In the sequel, we will apply it for $\Om' =B_r(x_\ep)$, $y =x_\ep$, $u = u_\ep$ and $f(u) =\frac1{\lambda_\ep} u e^{\alpha_\ep u^2} + \alpha u$. Noting that $v =-\Delta u_\ep$ and $F(u) = \frac1{2\al_\ep \lam_\ep} e^{\al_\ep u^2} + \frac{\alpha}2 u^2$. By Lemma \ref{Pohozaev}, we have
\begin{align}\label{eq:bb}
&\int_{B_r(x_\ep)} e^{\alpha_\ep u_\ep^2} dx \notag\\
&\quad\quad\quad = -\frac{\alpha \alpha_\ep \lam_\ep}{b_\ep^2} \int_{B_r(x_\ep)} (b_\ep u_\ep)^2 dx + \frac r4 \int_{\pa B_r(x_\ep)} e^{\alpha_\ep u_\ep^2} d\om + \frac{\alpha \alpha_\ep \lam_\ep}{b_\ep^2} \frac r4 \int_{\pa B_r(x_\ep)} (b_\ep u_\ep)^2 d\om\notag\\
&\quad\quad\quad\qquad + \frac{\al_\ep \lam_\ep}{4 b_\ep^2} r\int_{\pa B_r(x_\ep)} |\Delta(b_\ep u_\ep)|^2 d\om - \frac{\al_\ep \lam_\ep}{b_\ep^2} \int_{\pa B_r(x_\ep)} \frac{\pa (b_\ep u_\ep)}{\pa \nu} \Delta(b_\ep u_\ep) d\om\notag\\
&\quad\quad\quad\qquad -\frac{\al_\ep \lam_\ep}{2 b_\ep^2} r \int_{\pa B_r(x_\ep)} \lt(2 \frac{\pa \Delta (b_\ep u_\ep)}{\pa \nu} \frac{\pa(b_\ep u_\ep)}{\pa \nu} - \la \na \Delta (b_\ep u_\ep), \na (b_\ep u_\ep)\ra\rt) d\om. 
\end{align}
Using the representation of $G_\alpha$ in Lemma \ref{Greenfunctionlimit} and $x_\ep \to p$, we have
\[
\int_{B_r(x_\ep)} (b_\ep u_\ep)^2 dx = \int_{B_r(p)} G_\alpha(x,p)^2 dx + o_{\ep,r}(1) = o_r(1) + o_{\ep, r}(1),
\]
\[
r \int_{\pa B_r(x_\ep)} e^{\alpha_\ep u_\ep} d\om = o_r(1) + o_{\ep,r}(1),\quad r\int_{\pa B_r(x_\ep)} (b_\ep u_\ep)^2 d\om =o_r(1) + o_{\ep,r}(1),
\]
\[
r\int_{\pa B_r(x_\ep)} |\Delta(b_\ep u_\ep)|^2 d\om =r\int_{\pa B_r(p)} |\Delta G_\alpha(x,p)|^2 d\om + o_{\ep,r}(1) = \frac{\si^2}{8\pi^2} + o_r(1) + o_{\ep,r}(1),
\]
\[
\int_{\pa B_r(x_\ep)} \frac{\pa (b_\ep u_\ep)}{\pa \nu} \Delta(b_\ep u_\ep) d\om =\int_{\pa B_r(p)} \frac{\pa G_\alpha(x,p)}{\pa \nu} \Delta G_\al(x,p) d\om + o_{\ep,r}(1) =\frac{\si^2}{16 \pi^2} + o_r(1) + o_{\ep,r}(1),
\]
\[
r\int_{\pa B_r(x_\ep)} \frac{\pa \Delta (b_\ep u_\ep)}{\pa \nu} \frac{\pa(b_\ep u_\ep)}{\pa \nu} d\om =r\int_{\pa B_r(p)}  \frac{\pa \Delta (G_\alpha)}{\pa \nu} \frac{\pa(G_\alpha)}{\pa \nu} + o_{\ep,r}(1) = -\frac{\si^2}{8 \pi^2} + o_r(1) + o_{\ep,r}(1),
\]
and
\[
r\int_{B_r(x_\ep)} \la \na \Delta (b_\ep u_\ep), \na (b_\ep u_\ep)\ra d\om =r\int_{B_r(p)}\la \na \Delta G_\alpha, \na G_\alpha \ra d\om + o_{\ep,r}(1) =-\frac{\si^2}{8 \pi^2} + o_r(1) + o_{\ep,r}(1),
\]
where $o_{\ep,r}(1)$ and $o_r(1)$ mean that $\lim_{\ep\to 0} o_{\ep,r}(1) =0$ when $r$ is fixed, and $\lim_{r\to 0} o_r(1) =0$ respectively. Hence, we get
\begin{equation}\label{eq:abba}
\int_{B_r(x_\ep)} e^{\alpha_\ep u_\ep^2} dx = \frac{\lambda_\ep}{b_\ep^2} (\si^2 + o_r(1) + o_{\ep,r}(1)) + o_r(1) + o_{\ep,r}(1).
\end{equation}
We claim that $\sigma^2 >0$. Indeed, if this is not true, then $\sigma^2 =0$, and we have
\[
\int_{B_r(x_\ep)} e^{\alpha_\ep u_\ep^2} dx = \frac{\lambda_\ep}{b_\ep^2} (o_r(1) + o_{\ep,r}(1)) + o_r(1) + o_{\ep,r}(1).
\]
Since $|\Delta u_\ep|^2 dx \rightharpoonup  \de_p$ in the measure sense, then for any $r >0$ with $B_r(p)\subset \Om$, by using Adams inequality and cut-off function argument, we have
\begin{equation}\label{eq:xxxxx}
\lim_{\ep\to 0} \int_{\Om\setminus B_r(p)} e^{\alpha_\ep u_\ep^2} dx = |\Om| - |B_r(p)|.
\end{equation}
Fix a $r_0 >0$ such that $B_{2r_0}(p)\subset \Om$ and $|o_{r}(1)| \leq 1/4$ for any $r \leq 2r_0$. Choosing $\ep_0 >0$ such that $|x_\ep -p| < r_0$ and $|o_{\ep,2r_0}(1)|\leq 1/4$ for any $\ep \leq \ep_0$. Thus $B_{r_0}(p) \subset B_{2r_0}(x_\ep)$, and hence
\[
\limsup_{\ep\to 0} \int_{\Om\setminus B_{2r_0}(x_\ep)} e^{\alpha_\ep u_\ep^2} dx \leq |\Om| -|B_{r_0}(p)|\leq |\Om|.
\]
By H\"older inequality, we have
\begin{align*}
\frac{\lam_\ep}{b_\ep^2}  = \frac{\lt(\int_{\Om} |u_\ep| e^{\alpha_\ep u_\ep^2}dx\rt)^2}{\int_\Om u_\ep^2 e^{\alpha_\ep u_\ep^2} dx}\leq \int_\Om e^{\alpha_\ep u_\ep^2} dx\leq \frac12 \frac{\lambda_\ep}{b_\ep^2}  + \frac12 + \int_{\Om\setminus B_{2r_0}(x_\ep)} e^{\alpha_\ep u_\ep^2} dx.
\end{align*}
Thus
\[
\limsup_{\ep\to 0} \frac{\lambda_\ep}{b_\ep^2} \leq 1 + 2|\Om| < \infty.
\]
This together the estimates above and $\sigma^2 =0$ implies
\[
\lim_{r\to 0}\lim_{\ep \to 0} \int_{B_r(x_\ep)} e^{\alpha_\ep u_\ep^2} dx = 0.
\]
For any $r >0$ such that $B_{2r}(p) \subset \Om$, we then have $B_{r/2}(p) \subset B_r(x_\ep) \subset B_{2p}(p)$ for sufficiently small $\ep >0$. Thus by \eqref{eq:xxxxx}, it holds
\begin{equation}\label{eq:yyyyy}
\lim_{r\to 0}\lim_{\ep \to 0} \int_{\Om\setminus B_r(x_\ep)} e^{\alpha_\ep u_\ep^2} dx = |\Om|.
\end{equation}
Finally, we get
\[
\lim_{\ep \to 0} \int_{\Om} e^{\alpha_\ep u_\ep^2} dx \leq |\Om|,
\]
which is impossible. Then we must have $\sigma^2 >0$. This together \eqref{eq:abba} and \eqref{eq:yyyyy} yields
\begin{equation}\label{eq:gioihan}
\lim_{\ep\to 0} \int_\Om e^{\alpha_\ep u_\ep^2} dx = |\Om| + \sigma^2 \lim_{\ep\to 0} \frac{\lam_\ep}{b_\ep^2}.
\end{equation}
We can further locate $\sigma$ as follows.
\begin{lemma}\label{sigmaequal1}
It holds $\sigma =1$.
\end{lemma}
\begin{proof}
We know from Lemma \ref{Greenfunctionlimit} that $b_\ep u_\ep \to G_\alpha(\cdot,p)$ in $C^4_{\rm loc}(\o{\Om}\setminus\{p\})$ with
\[
G_\alpha(x,p) = -\frac{\sigma}{8\pi^2} \ln |x-p| + A_p + \psi(x),\quad \psi \in C^3(\o{\Om}),\, \psi(p) =0.
\]
We also know that $\sigma \not =0$. Suppose $\sigma < 0$, then $G_\alpha(\cdot, p) \leq -C$ in $B_r(p)$ for some $r>0$ and $C >0$. Hence $u_\ep < 0$ in $B_r(p)\setminus\{p\}$ for $\ep$ small enough. In the other hand, by H\"older inequality, we have
\[
b_\ep \geq \frac{\int_\Om |u_\ep| e^{\al_\ep u_\ep^2}dx}{\int_\Om e^{\al_\ep u_\ep^2} dx} \geq 1-\frac{\int_{\{|u_\ep|\leq 1\}} e^{\al_\ep u_\ep^2} dx}{\int_\Om e^{\al_\ep u_\ep^2} dx},
\]
thus
\[
\liminf_{\ep\to 0} b_\ep \geq 1 -\frac{|\Om|}{\sup_{u\in H_0^2(\Om),\, \|u\|_{2,\alpha} =1} \int_\Om e^{32\pi^2 u^2} dx} >0.
\]
Lemma \ref{limitofvarphiepsilon} implies that $u_\ep >0$ on $B_{Rr_\ep}(x_\ep)$ for any fixed $R>0$ provided that $\ep >0$ is small enough (since $c_\ep \to \infty$). However when $\ep$ is small enough, we then have $B_{Rr_\ep}(x_\ep) \subset B_r(p)$. We thus get a contradiction on the sign of $u_\ep$, hence $\sigma >0$. Whence, $G_\alpha(\cdot, p) \geq C >0$ in $B_r(p) \setminus\{p\}$ for some $r>0$ and $C >0$, hence $u_\ep >0$ in $B_r(p)\setminus\{p\}$ for sufficiently small $\ep >0$, and then we have
\[
\int_{B_r(p)} |u_\ep| e^{\al_\ep u_\ep^2} dx = \int_{B_r(p)} u_\ep e^{\al_\ep u_\ep^2} dx.
\]
Since $|\Delta u_\ep|^2 dx \rightharpoonup \de_p$ in the measure sense, and $u_\ep \to 0$ in $L^s(\Om)$ for any $s >1$, then by using Adams inequality and cut-off function argument, we can show that
\[
\lim_{\ep\to 0} \int_{\Om\setminus B_r(p)} |u_\ep| e^{\alpha_\ep u_\ep^2} dx = 0.
\]
Obviously,
\[
\int_\Om |u_\ep| e^{\alpha_\ep u_\ep^2} dx \geq \int_{\{|u_\ep|\geq1\}}  e^{\alpha_\ep u_\ep^2} dx =\int_\Om e^{\alpha_\ep u_\ep^2} dx - \int_{\{|u_\ep|\leq1\}}  e^{\alpha_\ep u_\ep^2} dx
\]
hence by \eqref{eq:estimate2} and Lebesgue dominated convergence theorem we have
\[
\liminf_{\ep \to 0}\int_\Om |u_\ep| e^{\alpha_\ep u_\ep^2} dx \geq \sup_{u\in H_0^2(\Om),\, \|u\|_{2,\alpha} =1} \int_\Om e^{32\pi^2 u^2} dx - |\Om| >0.
\]
Since
\[
|\sigma -1| = \lt|\lim_{\ep \to 0} \frac{\int_\Om u_\ep e^{\al_\ep u_\ep^2} dx}{\int_\Om |u_\ep| e^{\al_\ep u_\ep^2} dx} -1\rt| \leq 2 \lim_{\ep \to 0} \frac{\int_{\Om\setminus B_r(p)} |u_\ep| e^{\al_\ep u_\ep^2} dx}{\int_\Om |u_\ep| e^{\alpha_\ep u_\ep^2} dx} =0.
\]
Hence $\sigma =1$.
\end{proof}
To summarize, we have the following result.
\begin{lemma}\label{rewriteGreen}
$b_\ep u_\ep \rightharpoonup G_\alpha(\cdot,p)$ weakly in $H_0^{2,r}(\Om)$ for any $1< r < 2$ with
\begin{equation*}
\begin{cases}
\Delta^2 G_\alpha(\cdot,p) = \de_p + \alpha G_\alpha(\cdot,p)&\mbox{in $\Om$}\\
G_\alpha(\cdot,p) = \frac{\pa G_\alpha(\cdot,p)}{\pa \nu} = 0&\mbox{on $\pa \Om$.}
\end{cases}
\end{equation*}
Furthermore, $b_\ep u_\ep \to G_\alpha(\cdot,p)$ in $C^4_{\rm loc}(\o{\Om}\setminus\{p\})$. Also we have
\[
G_\alpha(x,p) = -\frac{1}{8\pi^2} \ln\, |x-p| + A_p + \psi(x),
\]
where $A_p$ is constant depending on $p$ and $\alpha$, $\psi\in C^3(\o{\Om})$, with $\psi(p) =0$.
\end{lemma}
\section{Capacity estimates}
We follow the argument in \cite{LuYang}. Notice that in this section, we still assume that $u_\ep$ blows up and the blow-up point $p\in \Om$. We use capacity estimates to calculate the limit of $\lambda_\ep/b_\ep^2$ to estimate from above the supremum of the functional $\int_\Om e^{32\pi^2 u^2} dx$ over functions $u\in H_0^2(\Om)$ with $\|u\|_{2,\al} =1$ under the assumption that $u_\ep$ blows up. The technique of using capacity estimate applied to this kind of problems was discovery by Li \cite{Li2001} in dealing with Moser--Trudinger inequality of first order derivatives.

Let $u_\ep^*$ be the function constructed by Lu and Yang in section \S5 in \cite{LuYang}. The main properties of this function are that $u_\ep^* \in H^2(B_\de(x_\ep)\setminus B_{Rr_\ep}(x_\ep))$ and satisfies the boundary conditions
\begin{equation}\label{eq:boundarycondition}
\begin{cases}
u_\ep^*(x) = \frac1{b_\ep} \lt(\frac1{8\pi^2} \ln \frac1\de + A_p\rt) &\mbox{on $\pa B_\de(x_\ep)$,}\\
u_\ep^*(x) = c_\ep + \frac1{b_\ep} \varphi\lt(\frac{x-x_\ep}{r_\ep}\rt)&\mbox{on $\pa B_{Rr_\ep}(x_\ep)$,}\\
\frac{\pa u_\ep^*}{\pa \nu} = -\frac1{8\pi^2 \de b_\ep}&\mbox{on $\pa B_\de(x_\ep)$,}\\
\frac{\pa u_\ep^*}{\pa \nu} = \frac{1}{b_\ep r_\ep} \frac{\pa \varphi}{\pa \nu} \lt(\frac{x-x_\ep}{r_\ep}\rt)&\mbox{on $\pa B_{Rr_\ep}(x_\ep)$,}
\end{cases}
\end{equation}
and enery identity
\begin{equation}\label{eq:energyidentity}
\int_{B_\de(x_\ep)\setminus B_{Rr_\ep}(x_\ep)} |\Delta u_\ep^*|^2 dx = \int_{B_\de(x_\ep)\setminus B_{Rr_\ep}(x_\ep)} |\Delta u_\ep|^2 dx + \frac{o(1)}{b_\ep^2}.
\end{equation}
Now we start to derive the capacity estimates. Consider the variational problem
\[
i_{\de,R,\ep} = \inf \int_{B_\de(x_\ep)\setminus B_{Rr_\ep}(x_\ep)} |\Delta u|^2 dx
\]
where infimum takes all over functions belonging to $H^2(B_\de(x_\ep) \setminus B_{Rr_\ep}(x_\ep))$ with the same boundary conditions as $u_\ep^*$. It is well known (see \cite{LiN2007,Li2012}) that this infimum is attained by a bi-harmonic function $ \mathcal T$ which is defined in the annular domain $B_\de(x_\ep) \setminus B_{Rr_\ep}(x_\ep)$ with the same boundary condition as $u_\ep^*$. The explicit form of $\mathcal T$ is given by
\[
\mathcal T(x) = \mathcal A \ln |x-x_\ep| + \mathcal B |x-x_\ep|^2 + \mathcal C |x-x_\ep|^{-2} + \mathcal D,
\]
with the explicit values of $\mathcal A, \mathcal B$ was given in \cite{LuYang} (section \S5) by solving a linear system. Hence
\begin{equation}\label{eq:capacityvalue}
i_{\de,R,\ep} = 8\pi^2 \mathcal A^2 \ln \frac{\de}{Rr_\ep} + 32 \pi^2 \mathcal A \mathcal B (\de^2 -R^2r_\ep^2) + 32\pi^2 \mathcal B^2(\de^4 -R^4 r_\ep^4).
\end{equation}
By the same proof of Lemma $5.1$ in \cite{LuYang}, we conclude that
\begin{equation}\label{eq:gioihan0}
\lim_{\ep\to 0} \frac1{c_\ep^2} \ln \frac{\lambda_\ep}{c_\ep^2} =0.
\end{equation}
From the definition of $r_\ep$, we have
\begin{equation}\label{eq:logrepsilon}
\ln \frac{Rr_\ep}{\de} = \ln \frac{R}{\de} + \frac{\ln \frac{\lambda_\ep}{c_\ep^2} - \alpha_\ep c_\ep^2}{4}.
\end{equation}
According to the argument in \cite{LuYang} with the help of \eqref{eq:gioihan0} and \eqref{eq:logrepsilon} and using the explicit values of $\mathcal A$ and $\mathcal B$, we obtain
\begin{multline}\label{eq:1stterm}
8\pi^2 \mathcal A^2 \ln \frac\de{Rr_\ep} = \frac{32\pi^2}{\alpha_\ep} \Bigg(1 + \frac{2 \varphi(R) + R\varphi'(R) + \frac1{4\pi^2} \ln \de -2A_p -\frac1{8\pi^2}}{b_\ep c_\ep}\\
+ \frac{\ln \frac{\lambda_\ep}{c_\ep^2} + 8 + 4\ln \frac Rr}{\alpha_\ep c_\ep^2}+ O\lt(\frac1{c_\ep^4} \ln^2 \frac{\lam_\ep}{c_\ep^2}\rt) + o\lt(\frac1{b_\ep c_\ep}\rt)\Bigg)
\end{multline}
and
\begin{equation}\label{eq:2and3term}
32\pi^2 \mathcal A\mathcal B(\de^2 -R^2 r_\ep^2) = O\lt(\frac1{b_\ep c_\ep}\rt),\qquad 32\pi^2 \mathcal B^2(\de^4 -R^4 r_\ep^4) = O\lt(\frac1{b_\ep^2}\rt).
\end{equation}
Remark that \eqref{eq:1stterm} is exactly the formula $(5.12)$ in \cite{LuYang} with a mistake on the coefficient of $R\vphi'(R)$. We correct this mistake in \eqref{eq:1stterm}. From \eqref{eq:energyidentity} and definition of $i_{\de,R,\ep}$ we have
\begin{align}\label{eq:upperbound}
i_{\de,R,\ep}& \leq \int_{B_\de(x_\ep)\setminus B_{Rr_\ep}(x_\ep)} |\Delta u_\ep|^2 dx + \frac{o(1)}{b_\ep^2}\notag\\
&=1 + \alpha \|u_\ep\|_2^2 -\int_{\Om\setminus B_\de(x_\ep)} |\Delta u_\ep|^2 dx -\int_{B_{Rr_\ep}(x_\ep)} |\Delta u_\ep|^2 dx + \frac{o(1)}{b_\ep^2}\notag\\
&=1 -\frac{1}{b_\ep^2} \lt(\int_{\Om\setminus B_\de(p)} |\Delta G_\al|^2 dx + \int_{B_R(0)} |\Delta \varphi|^2 dx -\alpha \|G_\al\|_2^2\rt) + \frac{o(1)}{b_\ep^2},
\end{align}
here we use Lemma \ref{limitofvarphiepsilon} and Lemma \ref{Greenfunctionlimit}. By integration by parts, we have
\begin{equation}\label{eq:IbPGalpha}
\int_{\Om \setminus B_\de(p)} |\Delta G_\al|^2 dx =-\frac1{16\pi^2} -\frac1{8\pi^2} \ln \de + A_p + \al \|G_\al\|_2^2 + O(\de\ln \de).
\end{equation}
\eqref{eq:upperbound} together \eqref{eq:IbPGalpha} gives
\begin{equation}\label{eq:upperbound*}
i_{\de,R,\ep} \leq 1-\frac1{b_\ep^2} \lt(\int_{B_r(0)} |\Delta \varphi|^2 dx -\frac1{16\pi^2}-\frac1{8\pi^2} \ln \de + A_p + O(\de\ln \de)\rt) +\frac{o(1)}{b_\ep^2}.
\end{equation}
Plugging \eqref{eq:capacityvalue}, \eqref{eq:1stterm} and \eqref{eq:2and3term} into \eqref{eq:upperbound*} and using the fact $32\pi^2/\alpha_\ep >1$, we obtain
\begin{align}\label{eq:exclude1}
\frac{32\pi^2}{\alpha_\ep} &\Bigg(\frac{2 \varphi(R) + R\varphi'(R) + \frac1{4\pi^2} \ln \de -2A_p -\frac1{8\pi^2}}{b_\ep c_\ep} + \frac{\ln \frac{\lambda_\ep}{c_\ep^2} + 8 + 4\ln \frac Rr}{\alpha_\ep c_\ep^2}\Bigg)\notag\\
&\hspace{2cm} + O\lt(\frac1{c_\ep^4} \ln^2 \frac{\lam_\ep}{c_\ep^2}\rt) + O\lt(\frac1{b_\ep c_\ep}\rt)\notag\\
&\leq -\frac1{b_\ep^2} \lt(\int_{B_r(0)} |\Delta \varphi|^2 dx -\frac1{16\pi^2}-\frac1{8\pi^2} \ln \de + A_p + O(\de\ln \de)\rt) +\frac{o(1)}{b_\ep^2}.
\end{align}
Multiplying both sides of \eqref{eq:exclude1} by $\alpha_\ep c_\ep^2$, using the fact $b_\ep \leq c_\ep$ and making a simple calculation, we obtain
\begin{align*}
\lt[\frac{32\pi^2}{\alpha_\ep} + O\lt(\frac1{c_\ep^2} \ln \frac{\lam_\ep}{c_\ep^2}\rt)\rt] \ln \frac{\lam_\ep}{c_\ep^2}&\leq -\frac{\alpha_\ep c_\ep^2}{b_\ep^2} \lt(\int_{B_R(0)} |\Delta \varphi|^2 dx - \frac1{8\pi^2} \ln \de\rt) -\frac{32\pi^2}{\alpha_\ep} 4 \ln \frac R\de\\
&\qquad -32\pi^2 \frac{c_\ep}{b_\ep} \lt(2\vphi(R) + R \vphi'(R) + \frac1{4\pi^2} \ln \de\rt) + O\lt(\frac{c_\ep^2}{b_\ep^2}\rt).
\end{align*}
Notice that 
\[
\ln \frac{\lam_\ep}{c_\ep^2} = \ln \frac{\lam_\ep}{b_\ep^2} + \ln \frac{b_\ep^2}{c_\ep^2},\quad \frac{32\pi^2}{\alpha_\ep} = 1 + O(\ep).
\]
These equalities together \eqref{eq:gioihan0} and the previous inequality imply
 \begin{align}\label{eq:lowerbound}
\ln \frac{\lam_\ep}{b_\ep^2} &\leq -(1+ o(1)) \frac{\alpha_\ep c_\ep^2}{b_\ep^2} \lt(\int_{B_R(0)} |\Delta \varphi|^2 dx - \frac1{8\pi^2} \ln \de\rt)-(4+o(1)) \ln \frac R\de\notag\\
&\quad -(1+o(1))\frac{\alpha_\ep c_\ep}{b_\ep} \lt(2\vphi(R) + R \vphi'(R) + \frac1{4\pi^2} \ln \de\rt)-(1 +o(1)) \ln \frac{c_\ep^2}{b_\ep^2} + O\lt(\frac{c_\ep^2}{b_\ep^2}\rt).
 \end{align}
Notice that by \eqref{eq:estimate2} and \eqref{eq:gioihan} we have $\lim_{\ep\to 0} \lam_\ep/b_\ep^2 >0$, hence 
\[
\ln \frac{\lambda_\ep}{b_\ep^2} \geq -C_0,
\]
for some $C_0 >0$. If $\tau = \lim_{\ep\to o} \frac{c_\ep}{b_\ep} =\infty$ then $\varphi \equiv 0$ by Lemma \ref{limitofvarphiepsilon} which shows that
\[
\ln \frac{\lam_\ep}{b_\ep^2} \leq (4+ o(1)) \frac{c_\ep^2}{b_\ep^2}  \ln \de -(8+o(1)) \frac{c_\ep}{b_\ep} \ln \de -(4+o(1)) \ln \frac R\de + O\lt(\frac{c_\ep^2}{b_\ep^2}\rt).
\]
Hence for a fixed $R >0$, by choosing $\de >0$ sufficiently small, we have
\[
-C_0 \leq \ln \frac{\lam_\ep}{b_\ep^2} \leq 2 \frac{c_\ep^2}{b_\ep^2}  \ln \de -(8+o(1)) \frac{c_\ep}{b_\ep} \ln \de -(4+o(1)) \ln \frac R\de,
\]
which is impossible since the right hand side tends to $-\infty$ when $\ep \to 0$. This contradiction proves that $1 \leq \tau < \infty$. Whence $\ln\frac{\lam_\ep}{b_\ep^2}$ is also bounded from above by \eqref{eq:lowerbound}. Also by \eqref{eq:lowerbound} we have
\[
\ln \frac{\lam_\ep}{b_\ep^2} \leq \lt[4\lt(\tau -1\rt)^2 +o(1)\rt] \ln \de -(4+o(1))\ln R + (64\pi^2 + o(1))\tau (\vphi(R) + 2R\vphi'(R))+ O\lt(1\rt)
\]
which then implies $\tau =1$ since otherwise by choosing $\ep >0$ and  $\de >0$ small enough we would obtain a contradiction with $\ln \frac{\lambda_\ep}{b_\ep^2} \geq -C_0$. With $\tau =1$, then
\[
\vphi(x) = \frac1{16 \pi^2} \ln \frac1{1+ \frac{\pi}{\sqrt{6}} |x|^2}.
\]
In this situation, $b_\ep \sim c_\ep$ and the estimates in \eqref{eq:2and3term} are improved as (see formula $(5.21)$ in \cite{LuYang})
\begin{equation*}\label{eq:improve}
32\pi^2 \mathcal A\mathcal B(\de^2 -R^2 r_\ep^2) = o\lt(\frac1{c_\ep^2}\rt),\qquad 32\pi^2 \mathcal B^2(\de^4 -R^4 r_\ep^4) = o\lt(\frac1{c_\ep^2}\rt).
\end{equation*}
Consequently, \eqref{eq:exclude1} becomes
\begin{align}\label{eq:lowerbound*}
 &\Bigg(\frac{2 \varphi(R) + R\varphi'(R) + \frac1{4\pi^2} \ln \de -2A_p -\frac1{8\pi^2}}{c_\ep^2} + \frac{\ln \frac{\lambda_\ep}{c_\ep^2} + 8 + 4\ln \frac R\de}{32\pi^2 c_\ep^2}\Bigg)\notag + O\lt(\frac1{c_\ep^4} \ln^2 \frac{\lam_\ep}{c_\ep^2}\rt) \notag\\
&\leq -\frac{(1+o(1))}{c_\ep^2} \lt(\int_{B_r(0)} |\Delta \varphi|^2 dx -\frac1{16\pi^2}-\frac1{8\pi^2} \ln \de + A_p + O(\de\ln \de)\rt) + o\lt(\frac1{ c_\ep^2}\rt).
\end{align}
Multiplying both sides of \eqref{eq:lowerbound*} by $32\pi^2 c_\ep^2$ we get
\begin{align}\label{eq:last}
&(1+o(1))\ln \frac{\lam_\ep}{c_\ep^2} \notag\\
&\leq -32\pi^2\lt(2 \varphi(R) + R\varphi'(R) + \frac1{4\pi^2} \ln \de -2A_p -\frac1{8\pi^2}\rt) -8 -4 \ln \frac R\de \notag\\
&\quad -(32\pi^2 +o(1))\lt(\int_{B_r(0)} |\Delta \varphi|^2 dx -\frac1{16\pi^2}-\frac1{8\pi^2} \ln \de + A_p + O(\de\ln \de)\rt) + o(1)\notag\\
&= -32\pi^2(2 \varphi(R) + R\varphi'(R)) -(32\pi^2 +o(1))\int_{B_r(0)} |\Delta \varphi|^2 dx +32\pi^2 A_p \notag\\
&\quad -4\ln R -2 + o(1) (1-\ln \de) + O(\de \ln \de)
\end{align}
It was computed in \cite{LuYang} (see formula $(5.22)$) that
\[
\int_{B_R(0)} |\Delta \vphi|^2 dx = \frac1{16 \pi^2} \ln \lt(1+ \frac{\pi}{\sqrt{6}} R^2\rt) + \frac1{96\pi^2} + O(R^{-2}).
\]
Thus
\[
\int_{B_R(0)} |\Delta \vphi|^2 dx = \frac1{8\pi^2} \ln R + \frac1{16 \pi^2} \ln \frac{\pi}{\sqrt{6}} + \frac1{96\pi^2} + O(R^{-2}).
\]
It is easy to see that
\[
\vphi(R) = -\frac1{8\pi^2} \ln R - \frac1{16 \pi^2} \ln \frac{\pi}{\sqrt{6}} + O(R^{-2}),
\]
and
\[
R\vphi'(R) =-\frac1{8\pi^2} + O(R^{-2}).
\]
Plugging these estimates into \eqref{eq:lowerbound*}, we get
\begin{align*}
\lim_{\ep \to 0}\ln \frac{\lam_\ep}{c_\ep^2} & \leq \frac{5}3 + 32\pi^2 A_p + \ln \frac{\pi^2} 6.
\end{align*}
Thus we have proved 
\begin{equation}\label{eq:upperboundblowup}
\sup_{u \in H_0^2(\Om),\, \|u\|_{2,\alpha} =1} \int_\Om e^{32\pi^2 u^2}dx \leq |\Om| + \frac{\pi^2}{6} e^{\frac{5}3 + 32\pi^2 A_p}.
\end{equation}


\section{Nonexistence of boundary bubbles}
The main result of this section is that the boundary bubbles do not occur. Suppose without loss of generality that $c_\ep =u_\ep(x_\ep) = \max_{x\in \Om} |u_\ep| \to \infty$ and $x_\ep \to p\in \pa \Om$. Note that we have $u_\ep \rightharpoonup 0$ weakly in $H_0^2(\Om)$, $u_\ep \to 0$ strongly in $H_0^1(\Om)$, strongly in $L^s(\Om)$ for any $s >1$ and a.e., in $\Om$.

\begin{lemma}\label{tobien}
It holds $|\Delta u_\ep|^2 dx \rightharpoonup \de_p$ in the sense of measure.
\end{lemma}
\begin{proof}
Note that $\int_\Om |\Delta u_\ep|^2 dx = 1 + \alpha \int_\Om |u_\ep|^2 dx \to 1$. If the conclusion of this lemma is not true, then there is $r>0$ small enough such that
\[
\lim_{\ep\to 0} \int_{B_r(p)\cap \Om} |\Delta u_\ep|^2 dx = \eta < 1.
\]
Choosing $\phi$ be a cut-off function on $C^4(\o{\Om})$ such that $0\leq \phi \leq 1$, $\phi =1$ on $\Om \cap B_{r/2}(p)$, $\phi =0$ on $\Om \setminus B_r(p)$, and $|\na \phi|\leq 4/r$. Since $u_\ep \rightharpoonup 0$ weakly in $H_0^2(\Om)$ and $u_\ep \to 0$ strongly in $H_0^1(\Om)$, hence
\[
\limsup_{\ep\to 0} \int_{B_p(r)\cap \Om} |\Delta (\phi u_\ep)|^2 dx \leq \eta.
\]
This together Adams inequality and \eqref{eq:ELequation} shows that $\phi u_\ep\in H_0^2(\Om)$ is weak solution of $\Delta^2(\phi u_\ep) = f_\ep$ with $f_\ep$ is bounded in $L^s(\Om)$ for some $s >1$. Applying the standard regularity theory implies that $\phi u_\ep$ is bounded in $C^3(\o{\Om})$. In particular, $c_\ep$ is bounded which contradicts with our assumption \eqref{eq:assumptiononblowupsequence}.
\end{proof} 
Lemma \ref{tobien} proves that if there is a blow-up point on the boundary $\pa \Om$, then this is the unique blow-up point in $\o{\Om}$. We next prove a convergence for $b_\ep u_\ep$.
\begin{lemma}\label{tozeroweak}
It holds $b_\ep u_\ep \rightharpoonup 0$ in $H_0^{2,r}(\Om)$ for any $1< r < 2$.
\end{lemma}
\begin{proof}
By the same proof of Lemma \ref{awayp}, $b_\ep u_\ep$ is bounded in $H_0^{2,r}(\Om)$ for any $1< r < 2$. Hence there is $F\in H_0^{2,r}(\Om)$ such that $b_\ep u_\ep \rightharpoonup F$ in $H_0^{2,r}(\Om)$ and $b_\ep u_\ep \to H$ in $H_0^1(\Om)$. Using the same method in the proof of Lemma \ref{Greenfunctionlimit}, we get that $F$ solves $\Delta^2 F = \alpha F$ in $\Om$. Since $F \in H_0^{2,r}(\Om)$ for any $1< r< 2$, by the standard regularity theory, we have $F \in C^3(\o{\Om})$. However, $\alpha < \lambda_1(\Om)$, we must have $F\equiv 0$.
\end{proof}

Applying Pohozaev type identity (Lemma \ref{Pohozaev}) to equation \eqref{eq:ELequation} on the domain $\Om \cap B_r(p)$, we obtain by the same way in the estimates for $\sigma^2$ that 
\[
\lim_{\ep\to 0} \int_\Om e^{\alpha_\ep u_\ep^2} dx = |\Om|
\]
which contradicts with \eqref{eq:estimate2}. Therefore, the blow-up point $p$ can not lie on $\pa \Om$.

\section{Proof of Theorem \ref{Main1}}
Let $c_\ep, x_\ep, p$ and $A_p$ as before. We have shown in section \S3 that if blow-up occurs, i.e., $c_\ep \to \infty$ then the blow-up point $p$ lies in the interior of  $\Om$, and the supremum
\begin{equation}\label{eq:blowupcase}
\sup_{u \in H_0^2(\Om),\, \|u\|_{2,\alpha} =1} \int_\Om e^{32\pi^2 u^2}dx \leq |\Om| + \frac{\pi^2}{6} e^{\frac{5}3 + 32\pi^2 A_p}.
\end{equation}
We are in position to prove Theorem
\begin{proof}[Proof of Theorem]
If there exists a function $u_0\in H_0^2(\Om)$ such that $\|u_0\|_{2,\alpha} =1$ and
\[
\int_\Om e^{32\pi^2 u_0^2}dx =\sup_{u \in H_0^2(\Om),\, \|u\|_{2,\alpha} =1} \int_\Om e^{32\pi^2 u^2}dx,
\]
then our proof is finished. Otherwise, the blow-up case occurs, hence Theorem follows from \eqref{eq:blowupcase}.
\end{proof}
We finish this section by give a proof of Proposition \ref{imply} which shows that our inequality \eqref{eq:Maininequality} implies the one of Lu and Yang \eqref{eq:LuYang}.

\begin{proof}[Proof of Proposition \ref{imply}]
Let $a_1, a_2, \ldots,a_k$, $k\geq 1$ be the number such that $0\leq a_1 < \lam_1(\Om)$, $0\leq a_2 \leq \lam_1(\Om) a_1, \ldots, a_k\leq \lam_1(\Om) a_{k-1}$. It is easy to see that
\[
q(t) \leq 1+ a_1t + a_1 \lam_1(\Om) t^2 + \cdots + a_1 \lam_1(\Om)^{k-1} t^k.
\]
Denote $a = a_1/\lam_1(\Om) < 1$ and $p(t) = 1 + at+\cdots at^k$  then 
\[
q(t) \leq p(\lam_1(\Om) t).
\]
We claim that there exist $b \in (a,1)$ such that 
\begin{equation}\label{eq:claim**}
p(t) \leq \frac 1{1 -bt},\qquad \forall\, t\in [0,1].
\end{equation}
Indeed, this claim is equivalent to
\[
\frac{1-bt}b(1 +t +\cdots+ t^{k-1}) \leq \frac1a,\qquad \forall \, t\in [0,1].
\]
Note that
\[
\frac{1-bt}b(1 +t +\cdots+ t^{k-1}) = \frac{1-b}{b}(1+ t+ \cdots + t^{k-1}) + 1-t^k \leq k\frac{1-b}b + 1.
\]
Since $a < 1$, hence we can choose $b\in (a,1)$ such that \eqref{eq:claim**} holds.

Denote $\alpha = b\lambda_1(\Om)$ with $b$ is given in \eqref{eq:claim**}. For any $u\in H_0^2(\Om)$ such that $\|\Delta u\|_2 \leq 1$, then $\lambda_1(\Om) \|u\|_2^2 \leq 1$. By our claim \eqref{eq:claim**}, we have
\[
q(\|u\|_2^2) \leq p(\lambda_1(\Om) \|u\|_2^2) \leq \frac1{1 -\alpha \|u\|_2^2}.
\]
Let 
\[
v = \frac u{(1 -\alpha \|u\|_2^2)^{1/2}},
\]
then $\|v\|_{2,\al} \leq 1$ and $v^2 \geq q(\|u\|_2^2) u^2$. This together \eqref{eq:Maininequality} implies \eqref{eq:LuYang}.
\end{proof}

\section{Proof of Theorem \ref{Existence}}
In this section, we construct functions $\phi_\ep\in H_0^2(\Om)$ such that $\|\phi_\ep\|_{2,\alpha} =1$ and
\[
\int_{\Om} e^{32\pi^2 \phi_\ep^2} dx > |\Om| + \frac{\pi^2}{6} e^{\frac{5}3 + 32\pi^2 A_p}.
\]
This fact together \eqref{eq:upperboundblowup} shows that the blow-up case can not occur, and hence proves our Theorem.

Denote $r =|x-p|$. Recall that 
\[
G_\alpha(x,p) = -\frac1{8\pi^2} \ln r + A_p + \psi(x),\quad \psi \in C^3(\o{\Om}),\, \psi(p) =0.
\]
Following the construction in \cite{LuYang} (section \S7), let us define
\begin{equation}\label{eq:testfunction*}
\phi_\ep = 
\begin{cases}
c + \frac{a - \frac1{16\pi^2} \ln\lt(1 + \frac{\pi}{\sqrt{6}} \frac{r^2}{\ep^2}\rt)}c + \frac{A_p +\psi}c + \frac bc r^2,&\mbox{if $r \leq R\ep$,}\\
\frac1c G_\alpha&\mbox{if $r > R\ep$,}
\end{cases}
\end{equation}
where $a,b,c$ are constants determined later such that $\phi_\ep \in H_0^2(\Om)$ and $\|\phi_\ep\|_{2,\al} =1$. 

We choose $R = -\ln \ep$. To ensure that $\phi_\ep \in H_0^2(\Om)$, we choose $a,b,c$ such that
\[
\lim_{r\uparrow  R\ep} \phi_\ep = \lim_{r\downarrow R\ep} \phi_\ep,\quad \lim_{r\uparrow R\ep} \na \phi_\ep = \lim_{r\downarrow R\ep} \na \phi_\ep.
\]
The simple computation shows that
\begin{equation}\label{eq:abvalue}
\begin{cases}
a =-c^2+ \frac1{16\pi^2} \ln\lt(1 + \frac{\pi}{\sqrt{6}} R^2\rt) -\frac{\ln (R\ep)}{8\pi^2} -bR^2\ep^2, \\
b = -\frac1{16\pi^2 R^2 \ep^2\lt(1 + \frac{\pi}{\sqrt{6}} R^2\rt)}.\\
\end{cases}
\end{equation}
It was computed in \cite{LuYang} that
\[
\|\Delta \phi_\ep\|_2^2 = \frac1{16 \pi^2 c^2} \lt(\ln \frac\pi{\sqrt{6}\ep^2} + 16\pi^2 A_p -\frac 56\rt) + \frac{\al}{c^2} \|G_\al\|_2^2 + O\lt(\frac1{c^2 \ln^2 \ep}\rt).
\]
We have
\begin{align*}
\int\limits_{\Om\setminus B_{Rr_\ep}(p)} \phi_\ep^2 dx &=\frac1{c^2} \int\limits_{\Om\setminus B_{Rr_\ep}(p)} G_\alpha^2 dx = \frac1{c^2} \|G_\alpha\|_2^2 -\frac1{c^2}\int\limits_{B_{Rr_\ep}(p)} G_\al^2 dx = \frac1{c^2} \|G_\alpha\|_2^2 + \frac{O(\ep^4 \ln^6(\ep))}{c^2}.
\end{align*}
On $B_{R\ep}(p)$ we have
\begin{align*}
\phi_\ep(x) &= \frac1{16\pi^2c}\lt(\ln\lt(1 + \frac{\pi}{\sqrt{6}} R^2\rt) -\ln\lt(1 + \frac{\pi}{\sqrt{6}} \frac{r^2}{\ep^2}\rt)\rt) -\frac{\ln(R\ep)}{8\pi^2 c}\\
&\qquad\qquad + \frac{A_p +\psi}c -\frac bc R^2\ep^2\lt(1 -\frac{r^2}{R^2\ep^2}\rt),
\end{align*}
hence
\[
\int_{B_{R\ep}(p)} \phi_\ep^2 dx = \frac1{c^2} O(\ep^4 \ln^6 \ep).
\]
Combining all these estimates together, we get
\[
\|\phi_\ep\|_{2,\alpha}^2 = \frac1{16 \pi^2 c^2} \lt(\ln \frac\pi{\sqrt{6}\ep^2} + 16\pi^2 A_p -\frac 56\rt) + O\lt(\frac1{c^2 \ln^2 \ep}\rt).
\]
Thus we can choose $c$ such that $\|\phi_\ep\|_{2,\al} =1$ for $\ep$ small enough. Moreover, we have
\begin{equation}\label{eq:cvalue}
c^2 = \frac1{16 \pi^2} \lt(\ln \frac\pi{\sqrt{6}\ep^2} + 16\pi^2 A_p -\frac 56\rt) + O\lt(\frac1{\ln^2 \ep}\rt).
\end{equation}
We next compute $\int_\Om e^{32\pi^2 \phi_\ep^2}dx$. On $\Om\setminus B_{R\ep}(p)$ we have
\begin{align}\label{eq:AAA}
\int_{\Om\setminus B_{R\ep}(p)} e^{32\pi^2 \phi_\ep^2}dx &\geq \int_{\Om\setminus B_{R\ep}(p)} \lt(1 + \frac{32\pi^2}{c^2} G_\alpha^2\rt) dx\notag\\
&=|\Om| + \frac{32\pi^2}{c^2} \|G_\al\|_2^2 + O\lt(\frac1{\ln^2 \ep}\rt).
\end{align}
On $B_{R\ep}(p)$, using \eqref{eq:abvalue} and \eqref{eq:cvalue} we have
\begin{align*}
\phi_\ep^2 &\geq c^2 + 2\lt(a - \frac1{16\pi^2} \ln\lt(1 + \frac{\pi}{\sqrt{6}} \frac{r^2}{\ep^2}\rt) + A_p +\psi+ br^2\rt)\\
&=-c^2 + 2(a+c^2) -\frac1{8\pi^2} \ln\lt(1 + \frac{\pi}{\sqrt{6}} \frac{r^2}{\ep^2}\rt) + 2A_p + 2\psi + 2b r^2\\
&= -\frac1{16 \pi^2} \ln \frac\pi{\sqrt{6}\ep^2} +\frac 5{96\pi^2} +\frac1{8\pi^2} \ln\lt(1 + \frac{\pi}{\sqrt{6}} R^2\rt) -\frac{\ln (R\ep)}{4\pi^2}\\
&\qquad -\frac1{8\pi^2} \ln\lt(1 + \frac{\pi}{\sqrt{6}} \frac{r^2}{\ep^2}\rt) + A_p + O\lt(\frac1{\ln^2 \ep}\rt),
\end{align*}
here we use the fact $\psi (p) =0$, hence $\psi = O\lt(\frac1{\ln^2 \ep}\rt)$ on $B_{R\ep}(p)$ since $R = -\ln \ep$ and also $br^2 = O\lt(\frac1{\ln^2 \ep}\rt)$ on $B_{R\ep}(p)$. Hence, on $B_{R\ep}(p)$, we have
\begin{align*}
e^{32\pi^2 \phi_\ep^2} & \geq \lt(\frac{\pi^2}{6\ep^4}\rt)^{-1}e^{\frac 53 + 32\pi^2 A_p}\lt(1 + \frac{\pi}{\sqrt{6}} R^2\rt)^4 (R\ep)^{-8}\lt(1 + \frac{\pi}{\sqrt{6}} \frac{r^2}{\ep^2}\rt)^{-4} \lt(1+ O\lt(\frac1{\ln^2 \ep}\rt)\rt)\\
&= \frac{\pi^2}{6} e^{\frac 53 + 32\pi^2 A_p} \ep^{-4} \lt(1 + \frac{\pi}{\sqrt{6}} \frac{r^2}{\ep^2}\rt)^{-4}\lt(1+ O\lt(\frac1{\ln^2 \ep}\rt)\rt),
\end{align*}
since $R = -\ln \ep$. Integrating on $B_{R\ep}(p)$ and using a suitable change of variable, we get
\begin{align*}
\int_{B_{R\ep}(p)} e^{32\pi^2 \phi_\ep^2} dx &\geq \lt(1+ O\lt(\frac1{\ln^2 \ep}\rt)\rt)e^{\frac 53 + 32\pi^2 A_p} \int_{B_{\o{R}}(0)} (1 + |x|^2)^{-4} dx.
\end{align*}
with $\o R = \pi^{1/2} R/ 6^{1/4}$. Using polar coordinate we get
\begin{align*}
\int_{B_{\o{R}}(0)} (1 + |x|^2)^{-4} dx &= 2\pi^2 \int_0^{\o R} \frac{r^3}{(1+ r^2)^4} dr\\
&=\pi^2 \int_0^{\o R^2} \frac{r}{(1+r)^4} dr\\
&=\pi^2 \lt(\frac16 -\frac{1}{2(1+ \o R^2)^2} + \frac1{3(1+\o R^2)^3}\rt)\\
&=\frac{\pi^2}6 \lt(1 + O\lt(\frac1{\ln^4 \ep}\rt)\rt).
\end{align*}
Finally, we have
\begin{equation}\label{eq:AAAA}
\int_{B_{R\ep}(p)} e^{32\pi^2 \phi_\ep^2} dx \geq \frac{\pi^2}6 e^{\frac 53 + 32\pi^2 A_p} + O\lt(\frac1{\ln^2 \ep}\rt).
\end{equation}
Combining \eqref{eq:AAA} together \eqref{eq:AAAA} we obtain
\[
\int_\Om e^{32\pi^2 \phi_\ep^2} dx \geq |\Om| + \frac{\pi^2}6 e^{\frac 53 + 32\pi^2 A_p} + \frac{32\pi^2}{c^2} \|G_\al\|_2^2 + O\lt(\frac1{\ln^2 \ep}\rt).
\]
This together \eqref{eq:cvalue} imply that for $\ep$ is sufficiently small 
\[
\int_\Om e^{32\pi^2 \phi_\ep^2} dx > |\Om| + \frac{\pi^2}6 e^{\frac 53 + 32\pi^2 A_p},
\]
as our desire.

\section*{Acknowledgments}
This work is supported by CIMI postdoctoral research fellowship.


\begin{thebibliography}{99}

\bibitem{Adams1988}
D. R. Adams, \emph{A sharp inequality of J. Moser for higher order derivatives\text}, Ann. of Math., {\bf 128} (2) (1988) 385-398.

\bibitem{AD2004}
Adimurthi, and O. Druet, \emph{Blow-up analysis in dimension $2$ and a sharp form of Trudinger--Moser inequality\text}, Comm. Partial Differ. Equ., {\bf 29} (2004) 295--322.


\bibitem{Aubin}
T. Aubin, and Y. Y. Li, \emph{On the best Sobolev inequality\text}, J. Math. Pures Appl., {\bf 78} (1999) 353--387.




\bibitem{CC1986}
L. Carleson, and S. Y. A. Chang, \emph{On the existence of an extremal function for an inequality of J. Moser\text}, Bull. Sci. Math., {\bf 110} (1986) 113-127.




\bibitem{CohnLu2001}
W. S. Cohn, and G. Lu, \emph{Best constants for Moser-Trudinger inequalities on the Heisenberg group\text}, Indiana Univ. Math. J., {\bf 50} (2001) 1567-1591.

\bibitem{CohnLu2004}
W. S. Cohn, and G. Lu, \emph{Sharp constants for Moser-Trudinger inequalities on spheres in complex space $\C^n$\text}, Comm. Pure Appl. Math., {\bf 57} (2004) 1458-1493.

\bibitem{Csato2015}
G. Csat\'o, and P. Roy, \emph{Extremal functions for the singular Moser-Trudinger inequality in dimension two\text}, Calc. Var., {\bf 54} (2015) 2341--2366.

\bibitem{Csato2016}
G. Csat\'o, and P. Roye, \emph{Singular Moser–Trudinger inequality on simply connected domain\text}, Commun. in PDE, {\bf } (2016)

\bibitem{DS2004}
A. Dall'Acqua, and G. Sweers, \emph{Estimates for Green function and Poisson kernels of higher order Dirichlet boundary value problems\text}, J. Differ. Equa., {\bf 205} (2004) 466-487.




\bibitem{doO2014}
J. M. do \'O, and M. de Souza, \emph{A sharp Trudinger--Moser type inequality in $\R^2$\text}, Trans. Amer. Math. Soc., {\bf 366} (2014) 4513--4549.

\bibitem{doOSouza}
J. M. do \'O, and M. de Souza, \emph{A sharp inequality of Trudinger--Moser type and extremal functions in $H^{1,n}(\R^n)$\text}, J. Differ. Equ., {\bf 258} (2015) 4062--4101.

\bibitem{Druet}
O. Druet, and H. Emmanuel, F. Robert, \emph{Blow-up theory for elliptic PDEs in Riemannian geometry\text}, Math. Notes, vol. 45, Princeton University press, Princeton, NJ, 2004.

\bibitem{Flucher1992}
M. Flucher, \emph{Extremal functions for the Trudinger-Moser inequality in $2$ dimensions\text} 
Comment. Math. Helv., {\bf 67} (1992) 471--497

\bibitem{Fontana1993}
L. Fontana, \emph{Sharp borderline Sobolev inequalities on compact Riemannian manifolds\text}, Comment. Math. Helv., {\bf 68} (1993) 415--454.

\bibitem{FM2011}
L. Fontana, and C. Morpurgo, \emph{Adams inequalities on measure spaces\text}, Adv. Maths., {\bf 226} (2011) 5066--5119.

\bibitem{FM2015}
L. Fontana, and C. Morpurgo, \emph{Sharp Adams and Moser-Trudinger inequalities on $\R^n$ and other spaces of infinite measure\text}, preprint, arXiv:1504.04678v3.




\bibitem{KS2016}
D. Karmakar, and K. Sandeep, \emph{Adams inequality on the hyperbolic space\text}, J. Funct. Anal., {\bf 270} (2016) 1792-1817.



\bibitem{LamLu2012*}
N. Lam, and G. Lu, \emph{Sharp Moser--Trudinger inequality on the Heisenberg group at the critical case and applications\text}, Adv. Math., {\bf 231} (2012) 3259--3287.

\bibitem{LamLu2012d}
N. Lam, and G. Lu, \emph{Sharp Adams type inequalities in Sobolev spaces $W^{m,\frac nm}(\R^n)$ for arbitrary integer $m$\text}, J. Differential Equations, {\bf 253} (2012) 1143-1171.

\bibitem{LamLu2013}
N. Lam, and G. Lu, \emph{A new approach to sharp Moser--Trudinger and Adams type inequalities: a rearrangement--free argument\text}, J. Differential Equations, {\bf 255} (213) 298-325.

\bibitem{LiZhu}
Y. Y. Li, and M. Zhu, \emph{Sharp Sobolev trace inequalities on Riemannian manifolds with boundary\text}, Comm. Pure Appl. Math., {\bf 50} (1997) 449--487.



\bibitem{Li2001}
Y. Li, \emph{Moser--Trudinger inequaity on compact Riemannian manifolds of dimension two\text}, J. Partial Differ. Equa., {\bf 14} (2001) 163-192.


\bibitem{Li2005}
Y. Li, \emph{Extremal functions for the Moser-Trudinger inequalities on compact Riemannian manifolds\text}, Sci. China Ser. A, {\bf 48} (2005) 618–648.

\bibitem{LiN2007}
Y. Li, and C. Ndiaye, \emph{Extremal functions for Moser--Trudinger type inequality on compact closed $4-$manifolds\text}, J. Geom. Anal., {\bf 17} (2007) 669-699.

\bibitem{LiRuf2008}
Y. Li, and B. Ruf, \emph{A sharp Trudinger-Moser type inequality for unbounded domains in $\R^n$\text}, Indiana Univ. Math. J., {\bf 57} (2008) 451--480.

\bibitem{Li2012}
J. Li, Y. Li, and P. Liu, \emph{The $Q-$curvature on a $4-$dimensional Riemannian manifold $(M,g)$ with $\int_M Q dV_g = 8\pi^2$\text}, Adv. Math., {\bf 231} (2012) 2194--2223.

\bibitem{CSL1998}
C. Lin, \emph{A classification of solutions of conformally invariant fourth order equation in $\R^4$\text}, Comment. Math. Helv., {\bf 73} (1998) 203-231.

\bibitem{Lin1996}
K. Lin, \emph{Extremal functions for Moser's inequality\text}, Trans. Amer. Math. Soc., {\bf 348} (1996) 2663--2671.

\bibitem{Lions1985}
P. L. Lions, \emph{The concentration-compactness principle in the calculus of variations. The limit case. II\text}, Rev. Mat. Iberoamericana, {\bf 1} (1985) 45-121.



\bibitem{LuYang}
G. Lu, and Y. Yang, \emph{Adams' inequalities for bi-Laplacian and extremal functions in dimension four\text}, Adv. Maths., {\bf 220} (2009) 1135--1170.

\bibitem{LY2009}
G. Lu, and Y. Yang, \emph{The sharp constant and extremal functions for Moser--Trudinger inequalities involving $L^p$ norms\text}, Discrete Contin. Dyn. Syst., {\bf 25} (2009) 963--979.


\bibitem{M1970}
J. Moser, \emph{A sharp form of an inequality by N. Trudinger\text}, Indiana Univ. Math. J., {\bf 20} (1970/71) 1077-1092.

\bibitem{EM1993}
E. Mitidieri, \emph{A Rellich type identity and applications\text}, Commun. Partial Differential Equations, {\bf 18} (1993) 125-151.


\bibitem{O'Neil1963}
R. O'Neil, \emph{Convolution operators and $L(p,q)$ spaces\text}, Duke Math. J., {\bf 30} (1963) 129-142.


\bibitem{P1965}
S. I. Poho${\rm \check{z}}$aev, \emph{On the eigenfunctions of the equation $\Delta u + \lambda f(u) = 0$\text}, (Russian), Dokl. Akad. Nauk. SSSR, {\bf 165} (1965) 36-39.

\bibitem{Robert}
F. Robert, and J. Wei, \emph{Asymptotic behavior of a forth order mean field equation with Dirichlet boundary condition\text}, Indiana Univ. Math. J., {\bf 57} (2008) 2039--2060.

\bibitem{Ruf2005}
B. Ruf, \emph{A sharp Trudinger-Moser type inequality for unbounded domains in $\R^2$\text}, J. Funct. Anal., {\bf 219} (2005) 340--367.

\bibitem{RufSani2013}
B. Ruf, and F. Sani, \emph{Sharp Adams-type inequalities in $\R^n$\text}, Trans. Amer. Math. Soc., {\bf 365} (2013) 645--670.




\bibitem{TZ2000}
G. Tian, and X. Zhu, \emph{A nonlinear inequality of Moser--Trudinger type\text}, Calc. Var. Partial Differ. Equ., {\bf 10} (2000) 349-354.

\bibitem{Tintarev}
C. Tintarev, \emph{Trudinger--Moser inequality with remainder terms\text}, J. Funct. Anal., {\bf 266} (2014) 55--66.

\bibitem{T1967}
N. S. Trudinger, \emph{On imbedding into Orlicz spaces and some applications\text}, J. Math. Mech., {\bf 17} (1967) 473-483.


\bibitem{WY2012}
G. Wang, and D. Ye, \emph{A Hardy--Moser--Trudinger inequality\text}, Adv. Math., {\bf 230} (212) 294--320.

\bibitem{WeiXu1999}
J. Wei, and X. Xu, \emph{Classification of solutions of higher order conformally invariant equations\text}, Math. Ann., {\bf 313} (1999) 207-228.

\bibitem{Yang2006a}
Y. Yang, \emph{Extremal functions for a sharp Moser--Trudinger inequality\text}, Internat. J. Math., {\bf 17} (2006) 331–-338.

\bibitem{Yang2006}
Y. Yang, \emph{A sharp form of Moser--Trudinger inequality in high dimension\text}, J. Funct. Anal., {\bf 239} (2006) 100--126.

\bibitem{Yang2007}
Y. Yang, \emph{A sharp form of Moser--Trudinger inequality on a compact Riemannian surfaces\text}, Trans. Amer. Math. Soc., {\bf 359} (2007) 5761--5776.



\bibitem{Yang2015}
Y. Yang, \emph{Extremal functions for Trudinger-–Moser inequalities of Adimurthi-Druet type in dimension two\text}, J. Differ. Equ., {\bf 258} (2015) 3161--3193.

\bibitem{Yang2017}
Y. Yang, and X. Zhu, \emph{Blow-up analysis concerning singular Trudinger--Moser inequalities in dimension two\text}, J. Funct. Anal., in press.

\bibitem{Yang2017a}
Y. Yang, and X. Zhu, \emph{Extremal functions for singular Trudinger--Moser inequalities in the entire Euclidean space\text}, preprint, arXiv:1612.08247v1.

\bibitem{Y1961}
V. I. Yudovi${\rm \check{c}}$, \emph{Some estimates connected with integral operators and with solutions of elliptic equations\text}, (Russian), Dokl. Akad. Nauk. SSSR, {\bf 138} (1961) 805-808.


\bibitem{Zhu2014}
J. Zhu, \emph{Improved Moser--Trudinger inequality involving $L^p$ norm in $n$ dimensions\text}, Adv. Nonlinear Study, {\bf 14} (2014) 273--293.

\end{thebibliography}
\end{document}